\documentclass{amsart}
\usepackage{xspace}
\usepackage{plain}
\usepackage{verbatim}
\usepackage{epsfig,psfrag,subfigure}
\usepackage{color}
\usepackage{url}
\usepackage{hhline}
\usepackage{multirow}
\usepackage{xargs}
\usepackage{enumerate}
\RequirePackage{ifthen}
\RequirePackage{latexsym}
\RequirePackage{amsfonts}
\RequirePackage{amssymb}
\RequirePackage{amsthm}
\RequirePackage{amsmath}

\usepackage{./mydef,./boldfonts}
\chardef\atsign='100

\usepackage[numbers]{natbib}
\let\cite=\citet

\def\eps{\varepsilon}
\def\xHn#1{{\rm H}^{#1}}

\renewcommand{\bal}[1]{\color{black}#1\color{black}}

\newcommand{\norme}[2][]{\|#2\|_{#1}}

\renewcommand{\Hcurl}[1][\Omega]{\bH_{\mathrm{curl}}{\left(#1\right)}}
\renewcommand{\Hdiv}[1][\Omega]{\bH_{\mathrm{div}}{\left(#1\right)}}
\newcommand{\Hzcurl}[1][\Omega]{\bH_{\mathrm{0,curl}}{\left(#1\right)}}

\newcommand{\Lzmv}[1][2]{\dot{L}^2}
\def\bphi{\boldsymbol{\phi}}
\def\xCzero{{\rm C}^{0}}
\def\xLtwo{{\rm L}^{2}} 
\def\xCone{{\rm C}^{1}} 
\def\xHone{{\rm H}^{1}}
\def\xLinfty{{\rm L}^{\infty}} 

\def\eps{\varepsilon}

\newcommandx*{\Hdiveps}[3][1=\Omega,2=\epsilon,3={}]{\bH_{\mathrm{div}}^{#3}\left(#1,#2
  \right)}
\newcommandx*{\Hzdiveps}[3][1=\Omega,2=\epsilon,3={}]{\bH_{\mathrm{div=0}}^{#3}\left(#1,#2
  \right)}
\newcommandx*{\Hrcurl}[2][1=\Omega,2=r]{\bH_{\mathrm{curl}}^{#2}\left(#1\right)}
\newcommandx*{\Hzrcurl}[2][1=\Omega,2=r]{\bH_{\mathrm{0,curl}}^{#2}\left(#1\right)}

\begin{document}

\title[An IP method for the Maxwell equations in Heterogeneous Media]
{An Interior Penalty Method with $\xCzero$ Finite Elements for the
    Approximation of the Maxwell Equations in Heterogeneous Media:
    Convergence Analysis with Minimal Regularity}

\author[A. Bonito]{Andrea Bonito$^{1}$}

\address{Department of Mathematics, Texas A\&M University
3368 TAMU, College Station, TX 77843, USA.}

\email{bonito\atsign math.tamu.edu}

\author[J.-L. Guermond]{Jean-Luc Guermond$^{2}$}

\address{Department of Mathematics, Texas A\&M University 3368 TAMU,
  College Station, TX 77843, USA. On leave CNRS, France.}

\email{guermond\atsign math.tamu.edu}

\author[F. Luddens]{Francky Luddens$^{2,3}$}

\address{LIMSI, UPR 3251 CNRS, BP 133, 91403 Orsay cedex, France.}

\email{luddens\atsign limsi.fr}

\thanks{$^1$Partially supported by the NSF grant DMS-1254618.}
\thanks{$^2$Partially supported by the NSF grant DMS-1015984 .}

\thanks{$^3$Parts of this work was done during
  visits of Francky Luddens at Texas A\&M. The support of the University
  Paris-Sud 11 is acknowledged.}

\date{\today}

\begin{abstract} The present paper proposes and analyzes {an
    interior penalty technique} using $\xCzero$-finite elements to solve
  the Maxwell equations in domains with heterogeneous properties. The
  convergence analysis for {the boundary value problem and the
    eigenvalue problem} is done assuming only minimal regularity in
  Lipschitz domains. The method is shown to converge for any
  polynomial degrees and to be spectrally correct.
  \end{abstract}




\maketitle

\section{Introduction}
The objectives of the present paper is to propose and analyze a nodal
$\xCzero$-finite element technique to solve the Maxwell equations in
domains with heterogeneous properties. More precisely, given a
three-dimensional open domain $\Omega$ with boundary $\front$, we want to
construct an approximation of the following problem using an interior
penalty technique and $\xCzero$-Lagrange finite elements:
\begin{equation}\label{pre_e:strong} 
   \ROT(\kappa\ROT \bE) = \eps\bg,\quad \DIV(\eps \bE) =0, \qquad
   \bE\CROSS\bn_{|\partial\Omega}= 0,
\end{equation}
where the fields $\kappa$ and $\eps$ are only piecewise smooth.  This
task is non-trivial on two counts: first, the solution of
\eqref{pre_e:strong} is singular in general, see \eg
\cite{Bo_Gu_Lu_2012}; second, it is known since the pioneering work of
\cite{MR1112332} that $\bH^1$-conforming approximation techniques that
rely on uniform $\xLtwo$-stability estimates both on the curl and the
divergence of the approximate field do not converge properly if
$\Omega$ is non-smooth and non-convex. This defect is a consequence of
$\bH^1(\Omega)\cap \Hzcurl$ being a closed proper subspace of $\Hdiv
\cap{\Hzcurl}$. This is probably one reason why edge elements have
been favored over $\xCzero$-Lagrange finite elements over the years. It is
only recently, say since the ground-breaking ``rehabilitation'' work
of \cite{CosDau02}, \cite{BraPas04} and \cite{BraKolPas05} that
$\xCzero$-Lagrange finite elements have regained their status as credible
approximation tools for the Maxwell equations and more generally for
div-curl problems.  The key idea developed in the above references is
that the divergence of the discrete field approximating $\eps\bE$ must
be controlled in a space that is intermediate between $\xLtwo(\Omega)$
and $\xHn{-1}(\Omega)$. This program is carried out in \cite{CosDau02}
by controlling the divergence of $\eps\bE$ in a weighted $\xLtwo$-space
where the weight is a distance to the re-entrant corners of the domain
to some appropriate power depending on the strength of the
singularity. The analysis of the method by \cite{CosDau02} requires
the approximation space to contain the gradient of $\xCone$ scalar-valued
functions, which excludes low-order finite-elements spaces. This
restriction {on low-order elements} is removed in
\cite{BufCiaJrJam09} by considering a mixed form of the weighted
$\xLtwo$-stabilization technique {on special meshes}. The method
developed by \cite{BraPas04} and \cite{BraKolPas05} involves a
least-square approximation of a discrete problem with different test
and trial spaces. The trial space is $\bL^2(\Omega)$ and the
components of the test space are subspaces of $\xHone(\Omega)$.  The
numerical method uses piecewise constant functions for the trial space
and piecewise linear functions enriched with face bubbles for the test
space.  A technique based on a local $\xLtwo$-stabilization of the
divergence of $\eps\bE$ and using finite elements of order high enough
so as to contain the gradient of Argyris or Hsieh-Clough-Tocher
$\xCone$-finite elements is introduced in \cite{MR2485453}. The
convergence analysis of the method requires the source term to be
smooth enough so that $\ROT\bE \in \bH^r(\Omega)$ with $r>\frac12$.
This method is further revisited in two space dimensions in
\cite{MR2958336} to allow for low-order finite elements and to remove
the smoothness assumption on $\ROT\bE$.

The present paper is the second part of a research program started in
\cite{BoGu2011} {and is part of the PhD thesis of \cite{Thesis_Francky_2012}.}
The technique adopted in \citep{BoGu2011,GAFD_Giesecke_2010b} consists of
stabilizing the divergence of the field $\eps\bE$ in a negative
Sobolev norm through a mixed formulation. It is shown in
\citep{BoGu2011,GAFD_Giesecke_2010b} that stabilizing the divergence in $\xHn{-1}(\Omega)$
is sufficient to solve the boundary value problem
\eqref{pre_e:strong}, but it may not be sufficient in general to solve
the associated eigenvalue problem if only Lipschitz regularity of the
domain is assumed.  In this case the divergence must be stabilized in
$\xHn{-\alpha}(\Omega)$ with $\alpha\in (\frac{\ell}{2\ell-1},1]$ where
$\ell-1$ is the polynomial degree of the approximation of $\bE$,
$\ell\ge 1$.  Note in passing that the method introduced in
\citep{BoGu2011,GAFD_Giesecke_2010b} with the particular choice $\alpha=1$ has also been
proposed in \cite{MR2914268}. The convergence analysis of the boundary
value problem in \citep{MR2914268} assumes that the right-hand side is divergence
free and either the solution to \eqref{pre_e:strong} is smooth or the
degree of the finite element space is large enough or the mesh is
specifically constructed so as to contain the gradient of $\xCone$
scalar-valued functions. The method proposed in \cite{BoGu2011}
converges for all $\alpha\in (\frac{\ell}{2\ell-1},1]$ as stated in
\citep[Lemma~5.4]{BoGu2011}, and the convergence rate is even maximal
when $\alpha=1$ \emph{without} the extra assumptions used in
\citep{MR2914268}, provided the right-hand side of the boundary value
problem is solenoidal (which is usually the case).  Yet, the possibility
of choosing $\alpha<1$ has been introduced in \citep{BoGu2011} to
ensure the spectral correctness of the approximation for eigenvalue
problems.

The objective of the present paper is to generalize the analysis of
\cite{BoGu2011} to boundary and eigenvalue problems with coefficients
$\kappa$ and $\eps$ in \eqref{pre_e:strong} that are only piecewise
smooth.  Our analysis assumes only the natural regularity of the
solution; in particular the a priori regularity of $\bE$ may be lower
than that of $\bH^{\frac12}(\Omega)$, see Theorem
\ref{Thm:Hs_stability_bvp}. {We focus mainly our attention on the
convergence analysis in the very low regularity range, $\bE\in
\bH^{s}(\Omega)$, $0< s< \frac12$. This range is rarely investigated
in the literature since it entails many technical difficulties. One
purpose of the present paper is to show that these difficulties can be
handled properly when using continuous finite elements; the analysis
with discontinuous elements has already been done, see \eg
\cite{MR2263045,MR2324460}.}

{The approximation that we propose consists of using} a mixed
formulation with nodal finite elements and an interior penalty method
to account for the jumps in the coefficients $\kappa$ and
$\eps$. The convergence analysis presented holds for any polynomial degree (greater than one). 
One essential argument of this paper is the construction
of a smoothing operator in $\Hzcurl$ that commutes (almost) with the
curl operator, see Lemma~\ref{Lem:commute_K}.  In passing we correct a
mistake from \citep{BoGu2011} where the smoothing operator was not
constructed properly. The second important argument is
Lemma~\ref{l:cont_IP} in the Appendix.  This is a variant of Lemma~8.2
in \cite{MR2263045}; however, our proof slightly differs from that in
\cite{MR2263045} since the estimates therein do not seem to be uniform
in the meshsize.

\bal{The method presented in this paper has been implemented in a
  three-dimensional parallel MHD code, SFEMaNS, see \eg
  \cite{GLLNR11}.  The code is developed under an open source licence
  and is used to test various experimental and astrophysical dynamo
  scenarios, \eg
  \cite{GAFD_Giesecke_2010b,GiNoStGeLeHeLuGu_2012,HoNoMaVaLuLe_2013}.
  A substantial part of the work presented in this paper has been
  motivated by the VKS experiment, where the heterogeneous
  distribution of magnetic permeability plays a key role on the onset
  of the dynamo effect, see \cite{Monchaux07}.  SFEMaNs is also used
  to investigate MHD instabilities in liquid metal flows, see \eg
  \cite{He_No_Ca_Gu_2015} for an application to liquid metal
  batteries.}

The paper is organized as follows. We introduce notation and recall a
priori regularity results in \S\ref{Sec:Preliminaries}. The smoothing
operator in $\Hzcurl$ is introduced in \S\ref{s:approx_hcurl}.  The
key properties of this operator are stated in
Theorem~\ref{thm:approx_op} and Lemma~\ref{Lem:commute_K} (the
estimate \eqref{approx_delta_rot} is particularly important).  The
finite element framework and the interior penalty technique are
presented in \S\ref{s:ah}.  This section also contains stability
estimates for the weak formulation of the boundary value problem. The
convergence analysis for the boundary value problem is done in
\S\ref{Sec:Convergence_analysis_BVP}.  The two important results from
this section are Theorem~\ref{thm:cvg_norme_h_min_reg} and
Theorem~\ref{thm:cvg_norme_l2}. Theorem~\ref{thm:cvg_norme_h_min_reg}
establishes convergence in a discrete norm and
Theorem~\ref{thm:cvg_norme_l2} establishes convergence in
$\bL^2(\Omega)$ using a duality argument.  {Additional convergence
  estimates assuming full regularity are given in
  Theorem~\ref{thm:cvg_norme_h_smooth} for completeness.}  The
spectral correctness of the approximation of the eigenvalue problem is
analyzed in \S\ref{s:eigen}.  The paper is complemented with an
appendix containing technical details. Lemma~\ref{l:cont_IP} is one of
the key results from the Appendix.

\section{Preliminaries} \label{Sec:Preliminaries}
\subsection{Spaces}
Let $D$ be an open connected Lipschitz domain in $\Real^3$. (In the
rest of the paper $D$ denotes a generic open Lipschitz domain that may
differ from $\Omega$.)  The space of smooth functions with compact
support in $D$ is denoted $\calD(D)$.  The norm in $\xHone(D)$ is defined
as follows:
\begin{equation}
\|v\|_{\xHone(D)}^2:= \|v\|_{\xLtwo(D)}^2 + \|\GRAD v\|_{\bL^2(D)}^2.
\end{equation}
The space $\xHn{s}(D)$ for $s\in (0,1)$ is defined by the method of real
interpolation between $\xHone(D)$ and $\xLtwo(D)$ (see \eg
\cite{bkTartar_2007}), \ie
\begin{equation}
\xHn{s}(D) = [\xLtwo(D),\xHone(D)]_{s,2}.
\end{equation}
We also define the space $\xHone_0(D)$ to be the completion of $\calD(D)$
with respect to the following norm:
\begin{equation}
\|v\|_{\xHone_0(D)}:= \|\GRAD v\|_{\bL^2(D)}.
\end{equation}
This allows us again to define the space $\xHn{s}_0(D)$ for $s\in (0,1)$
by the method of real interpolation between $\xHone_0(D)$ and $\xLtwo(D)$ as follows:
\begin{equation}
\xHn{s}_0(D) = [\xLtwo(D),\xHone_0(D)]_{s,2}.
\end{equation}
This definition is slightly different from what is usually done; the
only difference occurs at $s=\frac12$. What we hereafter denote by
$\xHn{\frac12}_0(D)$ is usually denoted by $\xHn{\frac12}_{00}(D)$
elsewhere.  Owing to these definitions, the spaces $\xHn{s}_0(D)$ and
$\xHn{s}(D)$ coincide for $s\in [0,\tfrac12]$ and their norms are
equivalent \bal{but not uniform with respect to $s$ as $s$ approached
  $\frac12$}, (see \eg \cite[Thm 11.1]{LM68}, \bal{\cite[Thm.~1.4.2.4]{bkGrisvard} or}
\cite[Chap. 33]{bkTartar_2007}).  The space $\xHn{-s}(D)$ is defined
by duality with $\xHn{s}_0(D)$ for $0\le s\le 1$, \ie with a slight
abuse of notation we define
\begin{equation*}
\|v\|_{\xHn{-s}(D)} = \sup_{0\not=w\in \xHn{s}_0(D)}\frac{\int_D v w \dif \bx}{\|w\|_{\xHn{s}_0(D)}}.
\end{equation*}
It is a standard result that $\xHn{-s}(D)=[\xLtwo(D),\xHn{-1}(D)]_{s,2}$, see \bal{\cite[Thm.~3.1]{LM68}.}

The above definitions are naturally extended to the vector-valued
Sobolev spaces $\bH^s(D)$ and $\bH^s_0(D)$. We additionally introduce
the following spaces of vector-valued functions:
\begin{align}
  \Hcurl[D] &= \{\bv\in \bL^2(D)\ | \ \ROT \bv\in \bL^2(D) \}, \\
  \Hzcurl[D] &= \{\bv\in \bL^2(D)\ | \ \ROT \bv\in \bL^2(D), \
  \bv\CROSS\bn|_{\partial D} =0 \}, \\
  \Hrcurl[D] &= \{\bv\in \bL^2(D)\ | \ \ROT \bv\in \bH^r(D) \}, \\
 \Hzrcurl[D] &= \{\bv\in \bL^2(D)\ | \ \ROT \bv\in \bH^r(D), \
  \bv\CROSS\bn|_{\partial D} =0 \},
\end{align}
all equipped with their natural norm; for instance,
$\|\bv\|_{\Hcurl[D]}^2 = \|\bv\|_{\bL^2(D)}^2 +
\|\ROT\bv\|_{\bL^2(D)}^2$.

\subsection{The domain}
The domain $\Omega$ is a bounded open set in $\Real^3$, \bal{but 
  the analysis presented in this paper can be applied to the
  two-dimensional counterparts of the problem.}  The boundary of
$\Omega$, say $\front$, is assumed to have the Lipschitz regularity
and to be connected.  To simplify the presentation we also assume that
$0\in \Omega$ and $\Omega$ is star-shaped with respect to an open
neighborhood of $0$.   \bal{
This assumption implies the compact embedding stated in the following lemma.
\begin{lem}\label{lem:star_shaped}
  Let $\Omega$ be an open subset of $\Real^d$ for $d>0$. Then the
  following statements are equivalent: \begin{enumerate}[(i)]
  \item There exists a neighborhood of the origin $\mathcal V$ such
    that $\Omega$ is star-shaped with respect to all points in
    $\mathcal V$,
\item There exists $\chi>0$ such that, for any $\delta\in(0,1)$, 
\begin{equation}\label{eq:star_shape_neigh}
\Omega_\delta :=(1-\delta)\Omega + B(0,\delta\chi) \subset\subset\Omega,
\end{equation}
where the notation $ \subset\subset$ indicates a compact embedding.
\end{enumerate}
\end{lem}
\begin{proof}
  Equation~\eqref{eq:star_shape_neigh} immediately implies $(i)$ with
  $\mathcal V=B(0,\chi)$. Let us prove the converse 
  $(i)\Rightarrow(ii)$. Let $\mathcal V$ be a neighborhood of the
  origin and let us assume that $\Omega$ is star-shaped with respect
  to all points in $\mathcal V$.  In particular the following holds:
\begin{equation}\label{e:star}
(1-\delta)\Omega + \delta \calV \subset
\Omega \qquad  \forall \quad \delta\in [0,1].
\end{equation} 
Let $\chi >0$ such that $B(0,2\chi)\subset\mathcal V$ so that in
conjunction with \eqref{e:star}, we have
\begin{equation}\label{eq:nc_embedding}
(1-\delta)\Omega +
B(0,\delta \chi) \subset  (1-\delta)\Omega +
\delta B(0,2\chi) \subset (1-\delta)\Omega + \delta \calV \subset \Omega
\end{equation}
for all $\delta\in (0,1)$. 
Hence, it remains to prove that that the embedding is compact, which is done upon showing that $\overline{\Omega_\delta}\subset\Omega$.
To do this, let  $\{ x_n \}_{n\geq 0} \subset \Omega_\delta$ be a sequence converging to some $x\in\Real^d$ and write $x_n = (1-\delta)y_n + \delta\chi r_n$ for some $y_n\in\Omega$ and $r_n\in B(0,1)$. 
Upon extracting subsequences (still indexed by $n$), there exists $r \in\overline{B(0,1)}$ and $y \in \overline{\Omega}$ such that  $r_n \to r$, $y_n \to y$ and therefore  $x = (1-\delta)y + \delta\chi r$.
As $y\in\overline{\Omega}$, we deduce that the set $B(y,\delta\chi)\cap\Omega$ is non-empty, \ie there exist $\tilde y\in\Omega$ and $z\in B(0,1)$ such that $y = \tilde y + \delta\chi z$.
Rewriting $x = (1-\delta)\tilde y + \delta\chi (r+(1-\delta)z)$, we realize that $x\in(1-\delta)\Omega + B(0,2\delta \chi)$, which yields $x\in\Omega$ (owing to ~\eqref{eq:nc_embedding}) and $\overline{\Omega}_\delta \subset \Omega$. This proves that the embedding $\Omega_\delta\subset\Omega$ is compact.
\end{proof}
}

A key piece of the convergence analysis of the method that we propose
in this paper is based on the existence of a family of smoothing
operators in $\Hzcurl$. This construction is discussed in detail in
\S\ref{s:approx_hcurl}.  The main challenge one encounters when
constructing this family of operators is to make it compatible with
the boundary condition and to commutes with the curl operator. The
purpose of \bal{the star-shape assumption is to make this construction
  possible and to simplify the presentation. It can be lifted for generic
  bounded Lipschitz domains by invoking Proposition 4.15 (or
  Proposition 4.19) from \cite{Hofmann_Mitrea_Taylor_2007}.  
  The results presented in this paper remain valid for any domain
  bounded Lipschitz domains.}

\subsection{Mixed formulation of the problem}
It will prove convenient to reformulate the original problem
\eqref{pre_e:strong} in mixed form to have a better control on the
divergence of the field $\eps \bE$. More precisely, from now on we
consider the following problem: Given a vector field $\bg$, find $\bE$
and $p$ such that
\begin{equation}
\ROT (\kappa\ROT \bE) +  \eps\GRAD p = \eps\bg; 
\quad \DIV(\eps \bE) =0,\quad \bE\CROSS\bn_{|\front}=0,
\quad p_{|\front}=0. \label{bvp}
\end{equation}
The two problems \eqref{bvp} and \eqref{pre_e:strong} are equivalent
if $\DIV(\eps \bg)=0$, since in this case $p=0$ in \eqref{bvp} (apply
the divergence operator to the first equation).

The scalar fields $\kappa$ and $\eps$ are assumed to be piecewise
smooth. More precisely we assume that $\Omega$ is partitioned into $N$
Lipschitz open subdomains $\Omega_1,\cdots,\Omega_N$ such that the
restrictions of $\kappa$ and $\eps$ to these subdomains are smooth.  To
better formalize this assumption we define
\begin{align}
\Sigma &:= \bigcup_{i\neq j}\partial\Omega_i\cap\partial\Omega_j,\\
  \mathrm{W}_\Sigma^{1,\infty}(\Omega) &:=\left\{ \nu\in \xLinfty(\Omega) \ | \
\GRAD(\nu_{|\Omega_i})\in \bL^{\infty}(\Omega_i),\; i=1,\cdots, N\right\}.
\end{align}
We refer to $\Sigma$ as the interface between the subdomains
$\Omega_i$. In the rest of the paper we assume that the fields $\eps$
and $\kappa$ satisfy the following properties: There exist $\eps_{\min}$,
$\kappa_{\min}>0$ such that
\begin{equation}
  \eps,\kappa\in
  \mathrm{W}^{1,\infty}_\Sigma(\Omega),\qquad \text{and} 
  \qquad  \eps \geq \eps_{\min}, \quad \kappa \geq\kappa_{\min}\quad \text{\ae in }\Omega. 
\label{Hyp:eps_mu}
\end{equation}
The following stability results proved in \cite{Bo_Gu_Lu_2012} play
important roles in the stability of the finite element method
developed in this paper:
\begin{thm} \label{Thm:Hs_stability_bvp}
  Assuming that $\eps\bg\in \bL^2(\Omega)$ and \eqref{Hyp:eps_mu},
  Problem \eqref{bvp} has a unique solution in $\Hzcurl\CROSS \xHone_0(\Omega)$. Moreover,
  there exist $c$, $\tau_\eps$ and $\tau_\kappa$, depending on $\Omega$
  and the fields $\eps$ and $\kappa$, so that
\begin{align}
  \|\bE\|_{\bH^s(\Omega)} &\le c \|\eps\bg\|_{\bL^2(\Omega)}, && \forall s \in [0,\tau_\eps),\label{E_in_Hs}\\
  \|\ROT \bE\|_{\bH^s(\Omega)} &\le c \|\eps\bg\|_{\bL^2(\Omega)}, &&
  \forall s \in [0,\tau_\kappa).\label{Rot_E_in_Hs} \\
\|\ROT(\kappa\ROT \bE)\|_{\bL^2(\Omega)} + \|\GRAD p\|_{\bL^2(\Omega)}&\le c  \|\eps\bg\|_{\bL^2(\Omega)},
\label{Rot_kappa_Rot_E_in_L2}
\end{align}
\end{thm}

\begin{remark}
In general the regularity indices $\tau_\eps$ and $\tau_\kappa$ are
smaller than $\frac12$ when the domain $\Omega$ is not convex and the
scalar field $\eps$ and $\kappa$ are discontinuous across $\Sigma$.
\end{remark}

\section{Smooth approximation in $\Hzcurl$}\label{s:approx_hcurl}%
We introduce in this section a smoothing operator in $\Hzcurl$ that
will be used to prove the convergence of the finite element
approximation. The key difficulty that we are facing is to find an
approximation that is compatible with the boundary condition in
$\Hzcurl$ and commutes with the curl operator. We essentially proceed
as in \cite{BoGu2011} but modify the argument to correct an incorrect
statement made therein. When invoking $\calC_h (A\bE)_\eps$ in the
proof of Lemma~3.3 in \cite{BoGu2011} it was incorrectly assumed that
$(A\bE)_\eps$ is in $\Hzcurl$, which is not the case. We resolve this
issue in the present construction by introducing an additional scaling
operator, $S^\delta_D$. {Some of the tools introduced in this
section are similar in spirit to those developed in
\cite{Schoberl_2008,Christiansen_Winther_mathcomp_2008,Arnold_Falk_Winther_2010}}

\subsection{Extension operator} \label{Sec:extension}
Let $D$ be an open Lipschitz domain in $\Real^3$.  For any
$\bF\in\bL^1(D)$, we denote $E_{D}\bF$ the extension of $\bF$ by $0$,
\ie
\begin{equation}\label{ext_by_zero}%
E_D\bF(\bx) = \left\{\begin{array}{cl}%
\bF(\bx) &\textnormal{ if }\bx\in D,\\
0 &\textnormal{ elsewhere.}%
\end{array}\right.
\end{equation}%
Let $\delta\in [0,\frac12]$, define $\bar\delta:=1-\delta$ and set
$D_\delta := \bar\delta D$.  \bal{Note that the assumption on $\delta$
  means that $\bar\delta\in [\frac12,1]$, \ie the quantity
  $\bar\delta^{-1}$ is uniformly bounded with respect to $\delta$;
  this observation will be used repeatedly. } We define the scaling
operator $S^\delta_D : \bL^1(D) \longmapsto \bL^1(D_\delta)$ by
\begin{equation}\label{dil_op}%
\forall\bF\in\bL^1(D),\ \forall \bx\in D_\delta,\quad 
(S^\delta_D \bF)(\bx) := \bF\left(\bx\bar\delta^{-1}\right).
\end{equation}
\begin{lem}\label{Lem:tilde_curl}
The following commuting properties hold:
\begin{align}
S^\delta_{\Real^3} E_D &= E_{D_\delta} S^\delta_{D} \label{SE_ES}\\
\partial_{x_i} (S^\delta_D \bF) &= \bar\delta^{-1}
S^\delta_D (\partial_{x_i}\bF),
&& \forall \bF\in \bL^1(D),\ \forall i=1,\ldots,d,  \label{partialS_Spartial} \\
\ROT (E_D\bF) &= E_{D}(\ROT \bF),&& \forall \bF\in
\Hzcurl[D],
\label{RotE_ERot}\\
 \GRAD (E_D\bF) &= E_{D}(\GRAD \bF),&& \forall \bF\in \bH^1_0(D).
\label{GradE_EGrad}
\end{align}
\end{lem}
\begin{proof} \eqref{SE_ES} is evident and \eqref{partialS_Spartial}
  is just the chain rule. We only prove \eqref{RotE_ERot} since the
  proof of \eqref{GradE_EGrad} is similar. Let $\bF$ be a member
  of $\Hzcurl[D]$, then the following holds:
\begin{align*}
\langle  \ROT (E_D\bF) ,\bphi\rangle = \int_{\Real^3} (E_D\bF) \SCAL \ROT \bphi
=  \int_{D} \bF \SCAL \ROT \bphi =
\int_{D} \ROT \bF \SCAL \bphi, \quad 
\forall \bphi \in \pmb{\calD}(\Real^3),
\end{align*}
where the last equality holds owing to $\bF$ being in $\Hzcurl[D]$.
Then 
\[
\langle \ROT (E_D\bF) ,\bphi\rangle = \int_{\Real^3}
E_{D}(\ROT \bF) \SCAL \bphi, \quad \forall \bphi \in
\pmb{\calD}(\Real^3),
\]
which proves the statement.
\end{proof}

\begin{lem}\label{Lem:cont_E_S} {The following holds}
  for all $r\in[0,1]$: (i) the linear operator
  $E_D:\bH_0^r(D)\rightarrow\bH_0^r(\Real^3)$ is bounded; (ii) the
  family of operators
  $\{S_{D}^{\delta}\}:\bH^r(D)\rightarrow\bH^r(D_\delta)$ is uniformly
  bounded with respect to $\delta\in [0,\frac12]$.
\end{lem}
\begin{proof}
It is a standard result that $E_D$ is a continuous operator from
$\bL^2(D)$ to $\bL^2(\Real^3)$ and from $\bH_0^1(D)$ to
$\bH_0^1(\Real^3)$, see \cite{02208228}. Then the first assertion follows directly from the
interpolation theory. For the second part, a scaling argument ensures
that $S_{D}^\delta$ is a continuous operator from
$\bL^2(D)$ to $\bL^2(D_\delta)$. Using~\eqref{partialS_Spartial},
we infer that it is also a continuous operator from $\bH^1(D)$
to $\bH^1(D_\delta)$. The conclusion follows from the interpolation
theory.
\end{proof}

Taking $r\in\left[0,\frac12\right)$ and using the fact that the spaces
$\bH_0^r(\Omega)$ and $\bH^r(\Omega)$ coincide (with equivalent
norms), we infer that there exists $c$ such that,
\begin{equation}
  \forall\bF\in\bH^r(\Omega),\qquad
  \|E_\Omega\bF\|_{\bH^r(\Real^3)}\leq c\,\|\bF\|_{\bH^r(\Omega)}.
\end{equation}
Moreover, using this inequality and the second part of
Lemma~\ref{Lem:cont_E_S} with $D=\Real^3$, we infer that
$S_{\Real^3}^{\delta}E_\Omega: \bH^r(\Omega)\rightarrow\bH^r(\Real^3)$
is a linear continuous operator, and there exists $c$, uniform in
$\delta$, such that
\begin{equation}\label{eq:cont_SE}
  \forall\bF\in\bH^r(\Omega),\qquad  
  \|S_{\Real^3}^{\delta}E_\Omega\bF\|_{\bH^r(\Real^3)}\leq c\,\|\bF\|_{\bH^r(\Omega)}.
\end{equation}
\begin{lem}
The following holds:
\begin{align}
\forall \bF\in \Hzcurl, \qquad  \ROT (S^\delta_{\Real^3}E_\Omega\bF) 
= \bar\delta^{-1} S^\delta_{\Real^3}E_\Omega(\ROT \bF). \label{Rot_S_E}
\end{align}
\end{lem}
\begin{proof} Let $\bF\in \Hzcurl$.  By \eqref{partialS_Spartial} we
  infer that
\[
\ROT (S^\delta_{\Real^3}E_\Omega\bF) = \bar\delta^{-1} S^\delta_{\Real^3}\ROT (E_\Omega\bF).
\]
Then \eqref{RotE_ERot} from Lemma~\ref{Lem:tilde_curl} implies
\[
\ROT (S^\delta_{\Real^3}E_\Omega\bF) = \bar\delta^{-1} S^\delta_{\Real^3}E_\Omega(\ROT \bF),
\]
since $\bF\in \Hzcurl$.  This concludes the proof.
\end{proof}

\begin{lem} \label{Lem:SERot} The linear operator $S^\delta_{\Real^3}
  E_\Omega :\Hzrcurl[\Omega] \longrightarrow \Hzrcurl[\Real^3] $ is
  bounded for all $r\in [0,\frac12)$.  More precisely there is $c$,
  uniform with respect to $\delta$, so that the following holds:
\begin{equation}
\|\ROT (S^\delta_{\Real^3} E_{\Omega} \bF)\|_{\bH^r(\Real^3)} \le c \|\ROT\bF\|_{\bH^r(\Omega)}.
\label{norm_ROT_S_E}
\end{equation}
\end{lem}
\begin{proof}
  The identity \eqref{Rot_S_E} implies that $S^\delta_{\Real^3}
  E_\Omega$ is a continuous map from $\Hzcurl[\Omega]$ to
  $\Hzcurl[\Real^3]$. Let r$\in [0,\frac12)$ and let $\bF$ be a member
  of $\Hzrcurl[\Omega]$. A simple scaling argument implies that
  $S^\delta_\Omega \bF$ is a member of $\Hzrcurl[\Omega_\delta]$.
  Since $\ROT S^\delta_\Omega \bF$ is in $\bH^r(\Omega)$ and $r\in
  [0,\frac12)$, the extension by zero is stable in $\bH^r(\Real^3)$,
  \ie $E_{\Omega_\delta}\ROT S^\delta_\Omega \bF$ is a member of
  $\bH^r(\Real^3)$ and there is a constant $c$, uniform with respect
  to $\bF$ and $\delta$, so that 
\begin{align*}
\|E_{\Omega_\delta}\ROT S^\delta_\Omega \bF\|_{\bH^r(\Real^3)} & \le c'
\|\ROT S^\delta_\Omega \bF\|_{\bH^r(\Omega_\delta)} = c'
\bar\delta^{-1} \|S^\delta_\Omega \ROT \bF\|_{\bH^r(\Omega_\delta)} \\
& \le c \|\ROT\bF\|_{\bH^r(\Omega)}.
\end{align*}
Then, applying \eqref{SE_ES} and \eqref{RotE_ERot} to
the above inequality gives
\begin{align*}
  \|\ROT (S^\delta_{\Real^3} E_{\Omega} \bF)\|_{\bH^r(\Real^3)} & =
  \|\ROT (E_{\Omega_\delta}S^\delta_\Omega \bF)\|_{\bH^r(\Real^3)} =
  \|E_{\Omega_\delta}\ROT S^\delta_\Omega \bF\|_{\bH^r(\Real^3)} \\
  & \le c\, \|\ROT\bF\|_{\bH^r(\Omega)},
\end{align*}
which concludes the proof.
\end{proof}
We now state a lemma that gives some important approximation
properties of the operator $\bF\mapsto S^\delta_{\Real^3} E_\Omega \bF$.
\begin{lem}\label{l:approx_ftilde_delta} 
  There exists $K_1$ \bal{such that the following
  holds for every $\bF\in\bH_0^r(\Omega)$ and for all $r\in [0,1]$:}
\begin{align}%
  \norme[\bH_0^s(\Omega)]{\bF-S^\delta_{\Real^3} E_\Omega \bF} \le
  K_1\delta^{r-s}\norme[\bH_0^{r}(\Omega)]{\bF}
  \qquad 0\le s\le r\le 1,\label{f_tilde_delta_Hs}
\end{align}
and for all $r\in [0,\frac12)$ there exists ${K_2}$, such that the
following holds every $\bF\in\Hzrcurl$:
\begin{align}
  \norme[\bH^s(\Omega)]{\ROT(\bF-S^\delta_{\Real^3} E_\Omega \bF)} &
  \le {K_2}\delta^{r-s}\norme[\bH^{r}(\Omega)]{\ROT\bF} && 0\le s\le r<
  \frac12. \label{f_tilde_delta_rot_Hs}
\end{align}
\end{lem}%
\begin{proof}
  We prove the first inequality \eqref{f_tilde_delta_Hs} by means of
  an interpolation technique. \bal{Let $\bF\in\bH_0^1(\Omega)$, then using}
  Lemma~\ref{Lem:tilde_curl} together with $d=3$, we have
\begin{align*}
  \norme[\bL^2(\Omega)]{\bF-S^\delta_{\Real^3} E_\Omega \bF} &\le
  \norme[\bL^2(\Omega)]{\bF} +\norme[\bL^2(\Omega)]{S^\delta_{\Real^3}
    E_\Omega \bF} \le
  \left(1+\bar\delta^{\frac{d}{2}}\right)\norme[\bL^2(\Omega)]{\bF}
  \le 2\norme[\bL^2(\Omega)]{\bF}.\\
  \norme[\bH_0^1(\Omega)]{\bF-S^\delta_{\Real^3} E_\Omega \bF} & =
  \norme[\bL^2(\Omega)]{\GRAD(\bF-S^\delta_{\Real^3} E_\Omega \bF)}\le
  \norme[\bL^2(\Omega)]{\GRAD\bF}
  +\norme[\bL^2(\Omega)]{\GRAD (S^\delta_{\Real^3} E_\Omega \bF)} \\
  & = \norme[\bL^2(\Omega)]{\GRAD\bF} + \bar\delta^{-1}
  \norme[\bL^2(\Omega)]{S^\delta_{\Real^3}\GRAD (E_\Omega\bF)} =
  \norme[\bL^2(\Omega)]{\GRAD\bF} + \bar\delta^{\frac{d}{2}-1}
  \norme[\bL^2(\Omega)]{E_\Omega\GRAD
    \bF}\\
  &=
  \left(1+\bar\delta^{\frac{d}{2}-1}\right)\norme[\bL^2(\Omega)]{\GRAD
    \bF} \le 2\norme[\bH_0^1(\Omega)]{\bF}.
\end{align*}

We now derive an estimate for the mapping $\bH_0^1(\Omega)\ni \bF
\mapsto \bF-S^\delta_{\Real^3} E_\Omega \bF\in \bL^2(\Omega)$.  The definition of
$S^\delta_{\Real^3} E_\Omega \bF$ implies that
\begin{align*}
  \norme[\bL^2(\Omega)]{\bF-S^\delta_{\Real^3} E_\Omega \bF}^2 &=
  \int_\Omega\left|(E_\Omega\bF)(\bx)
    -(E_\Omega\bF)\left(\bx \bar\delta^{-1}\right)\right|^2\textrm{ d}\bx \\
  &= \int_\Omega\left|\int_0^1\GRAD (E_\Omega\bF)
    \left((1-t)\bx+t\bx \bar\delta^{-1}\right)\SCAL\frac{\delta}{\bar\delta}\bx
    \textrm{ d}t\right|^2\textrm{ d}\bx \\
  &\le \int_\Omega \frac{\delta^2}{\bar\delta^2}|\bx|^2
  \int_0^1\left|\GRAD(E_\Omega\bF) \left((1-t)\bx+t\bx\bar\delta^{-1}\right)\right|^2\textrm{
    d}t\textrm{ d}\bx.
\end{align*}
Then, we introduce $M := \sup_{\bx\in\Omega}|\bx|$, and we apply
Fubini's theorem:
\begin{align*}
  \norme[\bL^2(\Omega)]{\bF-S^\delta_{\Real^3}E_\Omega\bF}^2 &\le
  M^2\frac{\delta^2}{\bar\delta^2}\int_0^1
  \int_\Omega\left|\GRAD(E_\Omega\bF) \left((1-t)\bx+t\bx\bar\delta^{-1}\right)\right|^2\textrm{
    d}\bx\textrm{d}t
\end{align*}
Using a change of variable in the inner integral, we finally obtain
\begin{align*}
  \norme[\bL^2(\Omega)]{\bF-S^\delta_{\Real^3}E_\Omega\bF}^2 &\le
  M^2\frac{\delta^2}{\bar\delta^2}\norme[\bL^2(\Omega_\delta)]{\GRAD(E_\Omega\bF)}^2
  \int_0^1 \left(\frac{\bar\delta}{\bar\delta +\delta t}\right)^3\textnormal{d}t
  \le M^2 \delta^2
  \bar\delta^{-2}\norme[\bL^2(\Real^3)]{\GRAD (E_\Omega\bF)}^2.
\end{align*}
Since $\bF\in \bH^1_0(\Omega)$, we have
$\norme[\bL^2(\Real^3)]{\GRAD(E_\Omega\bF)}=\norme[\bL^2(\Real^3)]{E_\Omega
  \GRAD\bF} =\|\GRAD\bF\|_{\bL^2(\Omega)}$.
Using now the assumption $\delta\le \frac12$, \ie
$\bar\delta^{-1}\le 2$, we finally deduce that
\begin{equation}
  \norme[\bL^2(\Omega)]{\bF-S^\delta_{\Real^3}E_\Omega\bF} 
  \le 2M\delta\norme[\bL^2(\Omega)]{\GRAD\bF}
  =2M\delta\norme[\bH^1_0(\Omega)]{\bF}.
\end{equation}
We now set $K_1:=\max(2,2M)$ and we have proven that
\begin{align*}
\norme[\bL^2(\Omega)]{\bF-S^\delta_{\Real^3}E_\Omega\bF} &\le K_1\norme[\bL^2(\Omega)]{\bF}, \\
\norme[\bL^2(\Omega)]{\bF-S^\delta_{\Real^3}E_\Omega\bF} &\le K_1\delta\norme[\bH_0^1(\Omega)]{\bF}, \\
\norme[\bH_0^1(\Omega)]{\bF-S^\delta_{\Real^3}E_\Omega\bF} &\le K_1\norme[\bH_0^1(\Omega)]{\bF}.
\end{align*}
We conclude that \eqref{f_tilde_delta_Hs} holds by using the
Lions-Peetre Reiteration Theorem. 

We now turn our attention to \eqref{f_tilde_delta_rot_Hs}.  Let $r\in
[0,\frac12)$ and consider $s \in [0,r]$. Let $\bF$ be a member of
$\Hzrcurl$.  We observe first that $S^\delta_{\Real^3}E_\Omega\bF$ is
also in $\Hzrcurl$ owing to Lemma~\ref{Lem:SERot}.  Using
  \eqref{Rot_S_E} gives
\begin{align*}
  \norme[\bH_0^s(\Omega)]{\ROT(\bF-S^\delta_{\Real^3}E_\Omega\bF)}
  &= \norme[\bH_0^s(\Omega)]{\ROT\bF-\bar\delta^{-1}S^\delta_{\Real^3}E_\Omega \ROT\bF}\\
  &\le \norme[\bH_0^s(\Omega)]{\ROT\bF-\bar\delta^{-1}\ROT\bF}
  +  \bar\delta^{-1}\norme[\bH_0^s(\Omega)]{\ROT\bF-S^\delta_{\Real^3} E_\Omega (\ROT \bF)}\\
  &\le \delta \bar\delta^{-1}\norme[\bH_0^s(\Omega)]{\ROT\bF} +
  K_1\bar\delta^{-1}\delta^{r-s}\norme[\bH_0^{r}(\Omega)]{\ROT\bF}.
\end{align*} 
Using $\delta<\frac12$, \ie $\bar\delta^{-1}\le 2$, we have
\[
  \norme[\bH_0^s(\Omega)]{\ROT(\bF-S^\delta_{\Real^3}E_\Omega\bF)} 
  \le 2(K_1 + \delta^{1-r+s})\delta^{r-s}\norme[\bH_0^{r}(\Omega)]{\ROT\bF},
\]
Remembering that $\bH^s(\Omega)$ and $\bH_0^s(\Omega)$ coincide for
$0\le {s\leq } r<\frac12$ and that their norm are equivalent, \bal{say $\|\ROT\bF\|_{\bH_0^r(\Omega)}\le c(r)
\|\ROT\bF\|_{\bH^r(\Omega)}$, the above
inequality yields \eqref{f_tilde_delta_rot_Hs}  with $K_2 = 2(K_1+1) c(r)$} since $1-r+s\ge 1-r >
0$ and $\delta\leq\frac12$.
\end{proof}%

\subsection{Smooth approximation} \label{Sec:smooth_approximation} We
now use the above extension operator $S^\delta_{\Real^3}E_\Omega$
together with a mollification to construct a smooth approximation
operator. For $\delta\in (0,\frac12)$, we set
\begin{equation}
\rho_\delta(\bx):=\delta^{-3}\rho(\bx/\delta),\ \textnormal{where
}\rho(\bx):=\left\{\begin{array}{ll}%
    \eta\exp\left(-\frac 1{1-|\bx|^2}\right), &\textnormal{if }|\bx| < 1, \\
    0, &\textnormal{if }|\bx| \ge 1,%
\end{array}\right.
\end{equation}
where $\eta$ is chosen so that $\int_{\polR^3}\rho(\bx) \dif \bx= 1$. We define a
family of approximation operators $\{\calK_\delta\}_{\delta>0}$ in the
following way:
\begin{equation}\label{approx_op}%
\calK_\delta\bF = \rho_{\delta\chi}\star (S^\delta_{\Real^3}E_\Omega \bF),
\qquad \forall \bF \in \bL^1(\Omega)
\end{equation}
where $\chi$ is the constant introduced in \eqref{eq:star_shape_neigh}.%
\begin{thm}\label{thm:approx_op}%
  $\calK_\delta\bF|_{\Omega}$ is in $\pmb{\calC}_0^{\infty}(\Omega)$
  for all $\bF\in\bL^1(\Omega)$. \bal{Let $\ell$ be a positive
    integer. There exists a constant $K$, possibly depending on
    $\ell$,} such that the following estimates hold for any
  $0<\delta<\frac12$:
\begin{align}
  \norme[\bH_0^s(\Omega)]{\bF-\calK_\delta\bF} &\le
  K\delta^{r-s}\norme[\bH_0^{r}(\Omega)]{\bF}
  && 0\le s\le r\le 1 \label{approx_delta} \\
  \norme[\bH^s(\Omega)]{\ROT\bF-\ROT\calK_\delta\bF} &\le
  K\delta^{r-s}\norme[\bH^{r}(\Omega)]{\ROT\bF}
  && 0\le s\le r< \tfrac12 \label{approx_delta_rot} \\
  \norme[\bH^{r}(\Omega)]{\calK_\delta\bF} & \le K
  \chi^{s-r}\delta^{s-r}\norme[\bH^{s}(\Omega)]{\bF} && \bal{0 \le s\le r  \le \ell},\;
  s<\tfrac 12 \label{approx_exp}
\end{align}
and all $\bF \in \bH^r_0(\Omega)$, all $\bF\in\Hzrcurl[\Omega][r]$, and all
$\bF\in \bH^r(\Omega)$, respectively.
\end{thm}%
\begin{proof}
  Owing to the properties of the mollification operator, we have
  $\calK_\delta\bF|_{\Omega}\in\pmb{\calC}^{\infty}(\Omega)$. We now prove
  that the support of $\calK_\delta\bF$ is compact in $\Omega$.  The
  definition of the convolution operation implies that the following
  holds for all $\bx\in\polR^3$:
\begin{align*}
  \calK_\delta\bF(\bx) 
&= \int_{\polR^3}(S^\delta_{\Real^3}E_\Omega\bF)(\by)\rho_{\delta\chi}(\bx-\by)\textrm{ d}\by
  =
  \int_{\bar\delta\Omega}\bF(\by/\bar\delta)\rho_{\delta\chi}(\bx-\by)\textrm{
    d}\by.
\end{align*}
If $\bx\notin \bar\delta\Omega + B(0,\delta\chi)$, then for all $\by\in
\bar\delta\Omega$, we have $\rho_{\delta\chi}(\bx-\by) = 0$ and then
$\calK_\delta\bF(\bx) = 0$. As a result, $\calK_\delta\bF$ is
supported in $\bar\delta\Omega + B(0,\delta\chi)$ which is compactly
embedded in $\Omega$ owing to the assumption
\eqref{eq:star_shape_neigh}. Hence,
$\calK_\delta\bF\in\pmb{\calC}_0^{\infty}(\Omega)$; in particular, we have
$\calK_\delta\bF\in\Hzcurl$. We now prove the estimates
\eqref{approx_delta} to \eqref{approx_exp}. Let us first consider 
$\bF\in \bH^r_0(\Omega)$.
Using that \bal{$S^\delta_{\Real^3}E_\Omega$ is stable from $\bH^s_0(\Omega)$ to
$\bH^s_0(\Real^3)$ together with standard}
approximation properties of the mollification operator 
we deduce that there exists a uniform constant $K_3>0$ so that
\begin{align*}
  \norme[\bH_0^s(\Omega)]{S^\delta_{\Real^3}E_\Omega\bF-\calK_\delta\bF} &\le K_3
  (\delta\chi)^{r-s}\norme[\bH_0^{r}(\Real^3)]{S^\delta_{\Real^3}E_\Omega\bF},
  && 0\le s\le r\le 1.
\end{align*}
Using the triangle inequality and Lemma \ref{l:approx_ftilde_delta} we have
\begin{align*}
\|\bF-\calK_\delta \bF\|_{\bH^s_0(\Omega)} & \le 
\|\bF-S^\delta_{\Real^3}E_\Omega\bF\|_{\bH^s_0(\Omega)} +
 \norme[\bH_0^s(\Omega)]{S^\delta_{\Real^3}E_\Omega\bF-\calK_\delta\bF}\\
 &\le K_1\delta^{r-s}\norme[\bH_0^{r}(\Omega)]{\bF} + K_3
  \chi^{r-s}\delta^{r-s}\norme[\bH_0^{r}(\Real^3)]{S^\delta_{\Real^3}E_\Omega\bF} \\
 &\le (K_1 + 2 K_3
  \chi^{r-s})\delta^{r-s}\norme[\bH_0^{r}(\Omega)]{\bF}.
\end{align*}
This proves \eqref{approx_delta} with $K= K_1+2 K_3$ since $\chi\le 1$
and $s\le r$. Let us now consider $\bF\in \Hzrcurl$. Using that $\ROT
\calK_\delta\bF = \rho_{\delta\chi} \star
\ROT(S^\delta_{\Real^3}E_\Omega\bF)$, we infer that
\begin{align*}
  \norme[\bH^s(\Omega)]{\ROT (S^\delta_{\Real^3}E_\Omega\bF- \calK_\delta\bF)} &\le
  K_3 (\delta\chi)^{r-s}\norme[\bH^{r}(\Real^3)]{\ROT (S^\delta_{\Real^3}E_\Omega\bF)}
  && 0\le s\le r
\end{align*}
Using the triangle inequality together with \eqref{norm_ROT_S_E},
Lemma \ref{l:approx_ftilde_delta}, and assuming that $r<\frac12$ we
have
\begin{align*}
  \norme[\bH^s(\Omega)]{\ROT (\bF- \calK_\delta\bF)} &\le
\norme[\bH^s(\Omega)]{\ROT (\bF-S^\delta_{\Real^3}E_\Omega\bF)} +
 \norme[\bH^s(\Omega)]{\ROT (S^\delta_{\Real^3}E_\Omega\bF- \calK_\delta\bF)}\\
&\le  
{K_2} \delta^{r-s} \norme[\bH^{r}(\Omega)]{\ROT\bF}
+K_3 (\delta\chi)^{r-s}\norme[\bH^{r}(\Real^3)]{\ROT (S^\delta_{\Real^3}E_\Omega\bF)} \\
&\le {\delta^{r-s}}({K_2} + {K_3}\chi^{r-s}) \norme[\bH^{r}(\Omega)]{\ROT\bF},
\end{align*}
which proves \eqref{approx_delta_rot} with $K= {K_2+ K_3}$ since
$\chi\le 1$ and $s\le r$. Let us finally assume that $\bF\in
\bH^r(\Omega)$. Using again the properties of the mollification
operator, \bal{we infer that there exists $K_4(\ell)$ such that}
\begin{align*}
  \norme[\bH^{r}(\Omega)]{\calK_\delta\bF} \le \norme[\bH^{r}(\Real^3)]{\calK_\delta\bF} & \le \bal {K_4}
  (\delta\chi)^{s-r}\norme[\bH^{s}(\Real^3)]{S^\delta_{\Real^3}E_\Omega\bF} && 0 \le s\le
 \bal{ r  \le \ell}.
\end{align*}
Applying~\eqref{eq:cont_SE}, we obtain~\eqref{approx_exp}.  Note that
the assumption $s<\frac12$ is required in order to ensure that $
S^\delta_{\Real^3}E_\Omega\bF\in\bH^s(\Real^3)$.
\end{proof}%
\begin{rem}%
  {In the rest of the paper we will use \eqref{approx_exp} without
  mentioning the coefficient $\chi^{s-r}$} in the right hand sides. Indeed, we will use
  the inequality with $r$ bounded from above by the polynomial degree
  of the approximation; as a result, $\chi^{s-r}$ is uniformly
  bounded.
\end{rem}%
We end this section by mentioning a key commuting property on
$\calK_\delta$.

\begin{lem}\label{Lem:commute_K}
The following holds for any $\bF\in {\Hzcurl}$:
\begin{equation}\label{eq:commute_K}
\bar\delta\ROT\calK_\delta\bF = \calK_\delta(\ROT\bF).
\end{equation}
\end{lem}
\begin{proof}
  Owing to the properties of the convolution, the following holds for
  any $\bF\in {\Hzcurl}$:
\begin{equation*}
  \ROT\calK_\delta\bF = \rho_{\delta\chi}\star\left(\ROT\left(S_{\Real^3}^\delta E_\Omega\bF\right)\right).
\end{equation*}
Applying~\eqref{Rot_S_E}, we infer that
\begin{align*}
\ROT\calK_\delta\bF &= \rho_{\delta\chi}\star\left(\bar\delta^{-1}S_{\Real^3}^\delta E_\Omega(\ROT\bF)\right) \\
&=\bar\delta^{-1}\rho_{\delta\chi}\star\left(S_{\Real^3}^\delta E_\Omega(\ROT\bF)\right)
 = \bar\delta^{-1}\calK_\delta(\ROT\bF).
\end{align*}
This completes the proof.
\end{proof}
\section{Finite Element Approximation of the boundary value problem}\label{s:ah}%
We introduce and study the stability properties of a Lagrange finite
element technique for solving the boundary value problem \eqref{bvp}.
\subsection{Finite Element Spaces}
\bal{We assume that the sub-domains $\Omega_i$, $i=1,\ldots, N$ are
  polyhedra.  Let $\{\calT_h\}_{h>0}$ be a shape regular sequence of
  affine meshes that we assume to be conforming in each sub-domain
  $\Omega_i$, \ie $\Sigma$ is partitioned by a set of interface cells.
  We additionally assume that either it is possible to extract from
  each mesh $\calT_h$ another one, say $\calG_h$, that is globally
  conforming and of equivalent typical mesh size (this assumption is
  obviously true if $\calT_h$ is globally conforming or if $\calT_h$ is obtained from
  $\calG_h$ after a few refinement step consisting of subdivisions),
  or each interface cell on one side of $\Sigma$ is the union of
  interface cells from the other side of $\Sigma$, the cardinal number
  of this union being a priori bounded by a fixed number.  An example
  of triangulation satisfying both geometric assumptions above is shown
  in Figure~\ref{Fig:ex_mesh}.
\begin{figure}[h]
\includegraphics[width=0.33\textwidth,clip]{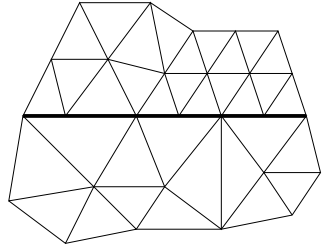}
\caption{Example of an admissible mesh. The interface $\Sigma$ is
  materialized by the thick line.}\label{Fig:ex_mesh}
\end{figure}
 We finally assume that the mesh
  sequence is quasi-uniform.  This assumption is non-essential and
  could be lifted by localizing some estimates; it is adopted here to
  simplify the presentation. The typical mesh size is denoted $h$.}
We introduce the following discrete space:
\begin{align}
  \bX_h &:= \left\{\bF\in \prod_{i=1}^N \pmb{\calC}^0(\bar\Omega_i),\ | \ 
    \forall K\in\calT_h,\;\bF_{|K}\in \pmb{\polP}_{\ell-1}\right\} \label{def_of_X_h}
\end{align}
where $\pmb{\polP}_{\ell-1}$ denotes the vector space of vector-valued
polynomial of total degree at most $\ell-1$, $\ell\ge 2$. Note that
the approximation space is non-conforming, \ie $\bX_h \not \subset
\Hzcurl$ and $\bX_h \not \subset \Hdiveps$. We assume that the mesh
sequence is such that there exists a family of local approximation operators
$\calC_h^l: \prod_{i=1}^N \bH^{\ell}(\Omega_i) \rightarrow \bX_h$
satisfying the following properties: there exists $c$ uniform in $h$
such that
\begin{align}
  \norme[\bH^{r}(\Omega_i)]{\calC_h^{l}\bF} 
&\le c\, \norme[\bH^{r}(\Omega_i)]{\bF}, && 0\le r < \tfrac32, \label{stab_Clement}\\
  \norme[\bH^t(\Omega_i)]{\calC_h^l\bF-\bF} &\le c\,
  h^{r-t}\norme[\bH^{r}(\Omega_i)]{\bF}, && 0\le t \le r\le \ell,\;
  t<\tfrac32, \label{approx_Clement}
\end{align}
for every $\bF\in \prod_{i=1}^N \bH^{\ell}(\Omega_i)$.  \bal{We
  introduce
  $\norme[\bL^2(\calT_h)]{.}^2:=\sum_{K\in \calT_h}
  \norme[\bL^2(K)]{.}^2$
  and
  $\norme[\bH^s(\calT_h)]{.}^2:=\sum_{K\in \calT_h}
  \norme[\bH^s(K)]{.}^2$.
  Owing to the quasi-uniformity assumption of the mesh sequence, we
  are going to regularly invoke various inverse inequalities like the following:
\begin{align} 
h  \norme[\bL^2(\calT_h)]{\ROT \kappa \ROT \bF_h} 
&\le c \norme[\bL^2(\calT_h)]{\kappa
    \ROT \bF_h},&& \forall \bF_h \in \bX_h\label{Inv_ineq1}\\
h^{s}\norme[\bH^s(\calT_h)]{\kappa \ROT\bF_h} &\leq c\norme[\bL^2(\calT_h)]{\kappa \ROT\bF_h},
&&\forall \bF_h \in \bX_h,\label{Inv_ineq2}\\
h^{\frac12} \norme[\bL^2(\Sigma\cup\front)]{\bF_h}
& \leq c \norme[\bL^2(\calT_h)]{\bF_h},&&\forall \bF_h \in \bX_h.\label{Inv_ineq3}
\end{align}
The assumptions adopted above for the mesh sequence imply that there
exists a family of discrete subspaces $\bY_h$ such that }
$\bY_h\subset \bX_h\cap \bH^1_0(\Omega)$ and a family of global
approximation operators
$\calC_h^g : \pmb{\calC}^\infty_0(\Omega) \longrightarrow \bY_h$ so
that
\begin{align}
  \norme[\bH^t(\Omega)]{\calC_h^g\bF-\bF} &\le c\,
  h^{r-t}\norme[\bH^{r}(\Omega)]{\bF}, && 0\le t \le r\le \ell,\;
  t<\tfrac32, \label{approx_Clement_smooth}
\end{align}
for every $\bF\in \pmb{\calC}^\infty_0(\Omega) $. 
 We additionally
introduce the scalar-valued discrete space
\begin{equation}
  M_h:=\left\{ q\in \xCzero(\bar\Omega),\ | \ 
    \forall K\in\calT_h,\; q_{|K}\in\polP_{\ell-1},\; q_{|\front}=0\right\} \subset \xHone_0(\Omega).
\end{equation}
Again, the approximation theory of finite elements
ensures that there exists an approximation operator $\calC_h^p:\xHone_0(\Omega)
\longrightarrow M_h$ satisfying the scalar counterparts of
\eqref{stab_Clement} and \eqref{approx_Clement} for all $q\in
\xHone_0(\Omega)\cap \xHn{l}(\Omega)$.
\begin{align}
  \norme[\xHn{l}(\Omega)]{\calC_h^p q} 
&\le c\norme[\xHn{l}(\Omega)]{q} && 0\le l\le \tfrac32 \label{scal_stab_Clement}\\
  \norme[\xHn{t}(\Omega)]{\calC_h^p q- q} &\le c
  h^{l-t}\norme[\xHn{l}(\Omega)]{q} && 0\le t \le l\le \ell,\;
  t<\tfrac32. \label{scalapprox_Clement}
\end{align}
\bal{Note that both $\bY_h$ and $M_h$ can be constructed either by
  invoking the existence of the mesh sequence $\{\calG_h\}_{h>0}$, or
  by constraining the possible hanging nodes on the interface
  $\Sigma$.}

We denote $\calF_h^i$ the set of the mesh interfaces: $F$ is an
interface if there are two elements in $\calT_h$, say $K_m$ and $K_n$
so that $F=K_m\cap K_n$ and $F$ is a $d-1$ manifold.  We denote
$\calF_h^\partial$ the set of the boundary faces: $F$ is a boundary
face if there is an element in $\calT_h$, say $K_m$ so that $F=K_m\cap
\front$ and $F$ is a $d-1$ manifold.  To simplify the notation we also
introduce $\calF_h:=\calF_h^i\cup \calF_h^\partial$.  For any mesh
interface $F\in \calF_h^i$, $F=K_m\cap K_n$, and any function $\bv$
whose restrictions over $K_m$ and $K_n$ are continuous, we define the
tangent and normal jump of $\bv$ across $F$ by
\begin{align}
\jump{\bv\times \bn}(\bx) &:= (\bv_{|K_m}\CROSS \bn_m)(\bx) +  (\bv_{|K_n}\CROSS \bn_n)(\bx),
\qquad \forall \bx \in F, \\
\jump{\bv\SCAL \bn}(\bx) &:= (\bv_{|K_m}\SCAL\bn_m)(\bx) +  (\bv_{|K_n}\SCAL \bn_n)(\bx),
\qquad \forall \bx \in F,
\end{align}  
where $\bn_l$ is the unit outer normal to $K_l$. 
The average of $\bv$ across across $F$ is defined by
\begin{equation}
\mean{\bv}(\bx) := \frac12\left(\bv_{|K_m}(\bx) +  \bv_{|K_n}(\bx)\right),
\qquad \forall \bx \in F.
\end{equation}
Whenever $F$ is a boundary face we set 
$\jump{\bv\CROSS \bn}(\bx) := \bv_{|K_m}\CROSS \bn_m(\bx)$,
$\jump{\bv\SCAL \bn}(\bx) := \bv_{|K_m}\SCAL \bn_m(\bx)$
and $\mean{\bv}(\bx) := \bv_{|K_m}(\bx)$.

\begin{rem}\label{rem:chf_smooth}
  Note that for any $\bF\in \pmb{\calC}^\infty_0(\Omega)$,
  $\calC_h^g\bF\in\bH_0^1(\Omega)$; in particular, we have $
  \jump{\calC_h^g\bF\CROSS\bn} = 0$ across all the interfaces in $\calF_h^i$.
\end{rem}

\subsection{Discrete formulation}
It will be useful to work with broken norms; for instance, we
introduce \bal{$\Omega_\Sigma:=\cup_{1\le i\le N}\Omega_i$ (recall
  that the domains $\Omega_i$ are open) together with} the following
notation:
\begin{align}
\|v\|_{\xHn{s}(\Omega_\Sigma)}^2&:=\sum_{i=1}^N \|v\|_{\xHn{s}(\Omega_i)}^2,\qquad
(v,w)_{\Omega_\Sigma}:= \sum_{i=1}^N \int_{\Omega_i} v w,\\
\|v\|_{\xLtwo(\Sigma\cup\front)}^2&:= \|v\|_{\xLtwo(\Sigma)}^2 +  \|v\|_{\xLtwo(\front)}^2,\qquad
(v,w)_{\Sigma\cup\front}:= \int_\Sigma v w + \int_\front v w.
\end{align}

We construct a discrete formulation of \eqref{bvp} by proceeding as in
\cite{BoGu2011}.  Let $\alpha \in [0,1]$ be a parameter yet to be
chosen.  We define the following bilinear form $a_h :\bX_h\CROSS
M_h\longrightarrow \Real$,
\begin{align}
  a_h((\bE_h,p_h)&,(\bF_h,q_h)) := \left(\kappa {\ROT
    \bE_h},{\ROT\bF_h}\right)_{\Omega_\Sigma} + \left(
  \mean{\kappa\ROT\bE_h}, \jump{\bF_h\CROSS \bn}
  \right)_{\Sigma\cup\front} \nonumber \\ &+ \theta \left(
  \mean{\kappa\ROT\bF_h}, \jump{\bE_h\CROSS
    \bn}\right)_{\Sigma\cup\front} + \gamma h^{-1} \left(
  \mean{\kappa} \jump{\bE_h\CROSS \bn}, \jump{\bF_h\CROSS
    \bn}\right)_{\Sigma\cup\front} \nonumber \\ &+\left(\eps\GRAD
  p_h,\bF_h\right)_\Omega - \left(\eps \bE_h,\GRAD q_h\right)_\Omega +
  c_\alpha \Big(h^{2\alpha}\left(\DIV(\eps \bE_h),
  \DIV(\eps\bF_h)\right)_{\Omega_\Sigma}\label{def_a_h} \\ &+
  h^{2(1-\alpha)}\left(\eps\GRAD p_h,\GRAD q_h\right)_\Omega +
  h^{(2\alpha-1)}\left(\jump{\eps\bE_h\SCAL\bn},\jump{\eps\bF_h\SCAL\bn}\right)_\Sigma\Big),
  \nonumber
\end{align}
where $\gamma$, $c_\alpha>0$, and $\theta\in\{-1,0,+1\}$ are
user-defined parameters.  We say that the formulation is
anti-symmetric, incomplete, or symmetric depending whether $\theta$ is
equal to $-1$, $0$, or $1$, respectively. The choice $\theta=1$
ensures the adjoint consistency of the method. The term proportional
to $\gamma$ enforces the weak continuity of the tangent component of
$\bE$.  The purpose of the term proportional to $c_\alpha$ is to
penalize $\DIV(\eps \bE_h)$ in $\bH^{-\alpha}(\Omega)$. The exponent
$\alpha$ is somewhat similar to the exponent that is used in
\cite{CosDau02} to define the $\xLtwo$-weighted space that controls
$\DIV(\eps \bE_h)$.

The discrete formulation considered in the rest of the paper consists
of looking for $(\bE_h,p_h)\in \bX_h{\times}M_h$ such that the
following holds for all $(\bF_h,q_h)\in\bX_h\CROSS M_h$:
\begin{equation}\label{e:discrete_formulation}
  a_h\left((\bE_h,p_h),(\bF_h,q_h)\right) 
  = \left(\eps\bg,\bF_h\right)_\Omega 
  + c_\alpha h^{2(1-\alpha)}\left(\eps\bg,\GRAD q_h\right)_\Omega,
\end{equation}
where $(\cdot,\cdot)_D$ henceforth denotes the scalar product in
$\xLtwo(D)$.

To perform the consistency analysis of the method we are led to introduce
\begin{align}
  \bZ^s(\Omega) =\{ \bF\in \Hzrcurl[\Omega][s]\st \ROT(\kappa\ROT\bF)\in \bL^2(\Omega),\
  \DIV(\eps\bF)\in \bL^2(\Omega)\}.
\end{align}
Owing to Theorem~\ref{Thm:Hs_stability_bvp}, it is a priori known that
there exists $s>0$ such that the solution to the boundary value
problem \eqref{bvp} is in $\bZ^s(\Omega){\cap\bH^s(\Omega)}$. We shall use the notation
$\bZ^s$ instead of $\bZ^s(\Omega)$ when {the context is unambiguous.}

\begin{prop} \label{Prop:extension_ah} Assuming \eqref{Hyp:eps_mu}, it
  is possible to extend the bilinear form $a_h(.,.)$ to
  $\left[(\bZ^s+\bX_h){\times}\xHone_0(\Omega)\right]^2$ for all $s>0$.
\end{prop}
\begin{proof}
  Note first that $M_h\subset \xHone_0(\Omega)$ and the extension of the bilinear
  form to scalar fields in $\xHone_0(\Omega)$ does not pose any difficulty.  We
  decompose $a_h$ into three pieces:
\begin{align*}
  a_{0h}((\bE_h,p_h)&,(\bF_h,q_h)) := \left( \mean{\kappa\ROT\bE_h},
    \jump{\bF_h\CROSS \bn} \right)_{\Sigma\cup\front}+ \theta \left(
    \mean{\kappa\ROT\bF_h},
    \jump{\bE_h\CROSS \bn}\right)_{\Sigma\cup\front}
  \\
  a_{1h}((\bE_h,p_h)&,(\bF_h,q_h)) := \left(\kappa {\ROT
      \bE_h},{\ROT\bF_h}\right)_{\Omega_\Sigma} + 
  \gamma h^{-1} \left( \mean{\kappa} \jump{\bE_h\CROSS \bn},
    \jump{\bF_h\CROSS \bn}\right)_{\Sigma\cup\front} \\
  a_{2h}((\bE_h,p_h)&,(\bF_h,q_h)) := c_\alpha
  \Big(h^{2\alpha}\left(\DIV(\eps \bE_h),
    \DIV(\eps\bF_h)\right)_{\Omega_\Sigma}+
  h^{2(1-\alpha)}\left(\eps\GRAD p_h,\GRAD q_h\right)_\Omega
  \\
  & +
  h^{(2\alpha-1)}\left(\jump{\eps\bE_h\SCAL\bn},\jump{\eps\bF_h\SCAL\bn}\right)_\Sigma\Big)
  + \left(\eps\GRAD p_h,\bF_h\right)_\Omega - \left(\eps \bE_h,\GRAD
    q_h\right)_\Omega.
\end{align*}
The bilinear form $a_{1h}$ can clearly be extended to
$\left[(\bZ^s+\bX_h){\times}\xHone_0(\Omega)\right]^2$, since every function $\bE$ in
$\bZ^s$ is such that $\jump{\bE\CROSS \bn}_{\Sigma\cup\front}$ is
zero. Hence, if either $(\bE,\bF)\in \bZ^s{\times}(\bZ^s+\bX_h)$  or 
$(\bE,\bF)\in (\bZ^s+\bX_h){\times}\bZ^s$, we set
\[
a_{1h}((\bE,p),(\bF,q)) := \left(\kappa {\ROT
    \bE},{\ROT\bF}\right)_{\Omega_\Sigma},
\]
for all $(p,q)\in \xHone_0(\Omega)$.  The bilinear form $a_{2h}$ can also be
extended to $\left[(\bZ^s+\bX_h){\times}\xHone_0(\Omega)\right]^2$, since every
function $\bE$ in $\bZ^s$ is such that $\jump{\eps\bE\SCAL
  \bn}_{{\Sigma}}$ is zero. Hence, if either $(\bE,\bF)\in
\bZ^s{\times}(\bZ^s+\bX_h)$ or $(\bE,\bF)\in
(\bZ^s+\bX_h){\times}\bZ^s$, we set
\begin{align*}
a_{2h}((\bE,p),(\bF,q)) &:= c_\alpha
  \Big(h^{2\alpha}\left(\DIV(\eps \bE),
    \DIV(\eps\bF)\right)_{\Omega_\Sigma}+
  h^{2(1-\alpha)}\left(\eps\GRAD p,\GRAD q\right)_\Omega\Big)
  \\
  & + \left(\eps\GRAD p,\bF\right)_\Omega - \left(\eps \bE,\GRAD
    q\right)_\Omega.
\end{align*}
for all $(p,q)\in \xHone_0(\Omega)$.  

The question of the extension of $a_{0h}$ is more subtle, and we must
now distinguish the trial and test spaces. We are going to use
Lemma~\ref{l:cont_IP} \bal{ to show that the bilinear form
  $(\bH^s(\Omega)\cap\Hcurl){\times}\bX_h\ni (\bphi,\bF_h) \longmapsto
  (\int_{F} \bphi\SCAL(\bF_{h|K_m}\CROSS\bn_m), \int_{F}
  \bphi\SCAL(\bF_{h|K_n}\CROSS\bn_n))\in \Real^2$
  is well defined for all $F=K_m\cap K_n\in \calF_h^i$, with the
  obvious equivalent statement if $F\in\calF_h^\partial$.} Let $\bE$
be a member of $\bZ^s$, then $\ROT\bE\in \bH^s(\Omega)$, $s>0$ and in
particular, $\ROT\bE \in \bH^\sigma(\Omega)$ for some
$\sigma \in (0,\frac12)$.  Owing to \eqref{Hyp:eps_mu},
{$\kappa \in \mathrm{W}_\Sigma^{1,\infty}(\Omega)$ so that
  $\kappa\ROT\bE\in \bH^{\sigma}(\Omega)$}, see \eg
\cite{Bo_Gu_Lu_2012}. Note in addition that $\bE$ being a member of
$\bZ^s$ implies that $\ROT(\kappa\ROT\bE)\in\bL^2(\Omega)$, which in
turn also implies that
$\mean{\kappa\ROT\bE}_{|\Sigma} =\kappa\ROT\bE_{|\Sigma}$. Hence,
Lemma~\ref{l:cont_IP} \bal{shows that the expressions
  $\int_{F} \kappa\ROT\bE\SCAL(\bF_{h|K_m}\CROSS\bn_m)$,
  $\int_{F} \kappa\ROT\bE\SCAL(\bF_{h|K_n}\CROSS\bn_n) $ are
  meaningful} for all $F\in \calF_h$ and for all
$(\bE,\bF_h)\in \bZ^s{\times}\bX_h$.  The extension of $a_{0h}$ for
$(\bE_h,\bF)\in \bX_h{\times}\bZ^s$ is justified similarly.  The
extension of $a_{0h}$ for $(\bE,\bF)\in \bZ^s{\times}\bZ^s$ is trivial
since the tangent jumps of $\bE$ and $\bF$ across $F$ are zero.
Summing up, $a_{0h}$ can be extended to
$\left[(\bZ^s+\bX_h){\times}\xHone_0(\Omega)\right]^2$ by setting
\begin{multline*}
  a_{0h}((\bE+\bE_h,p),(\bF+\bF_h,q)) := \left(\kappa\ROT\bE,
    \jump{\bF_h\CROSS \bn} \right)_{\Sigma\cup\front} +\left(
    \mean{\kappa\ROT\bE_h}, \jump{\bF_h\CROSS \bn}
  \right)_{\Sigma\cup\front} \\+ \theta \left( \kappa\ROT\bF,
    \jump{\bE_h\CROSS \bn}\right)_{\Sigma\cup\front} + \theta \left(
    \mean{\kappa\ROT\bF_h}, \jump{\bE_h\CROSS
      \bn}\right)_{\Sigma\cup\front},
\end{multline*}
for all $(\bE,\bE_h)\in \bZ^s{\times}\bX_h$, all $(\bF,\bF_h)\in
\bZ^s{\times}\bX_h$, and all $(p,q)\in \xHone_0(\Omega)$.
This ends the proof.
\end{proof}

\begin{remark}
One could avoid invoking Lemma~\ref{l:cont_IP} in the above proof by
using instead a result from \cite{BufCiaJr01I} where it is shown that the
bilinear form $\Hcurl{\times}\Hcurl\ni (\bphi,\bF) \longmapsto
\int_{F} \bphi\SCAL(\bF\CROSS\bn) \in \Real$ is well defined and
continuous for all $F\in \calF_h$
\end{remark}
\begin{rem}[Continuous Approximation of $p$]
  Observe that the approximation of the Lagrange multiplier $p$ is
  globally continuous.  This is critical to derive a global control of
  $\DIV (\eps \bE)$ in $\bH^{-\alpha}(\Omega)$ {(encoded in the bilinear form $a_{2h}$ in the above proof)} instead of
  $\prod_{i=1}^N\bH^{-\alpha}(\Omega_i)$. {We refer to \cite{BoGu2011} for more precisions.}
\end{rem}

\begin{lem} \label{Lem:consistency}
Assume~\eqref{Hyp:eps_mu} and let $(\bE,p)$ be the solution of
\eqref{bvp}. Let $s>0$ be such that $\bE\in\bZ^s$. The following holds
for any $(\bF+\bF_h,q)\in(\bZ^s+\bX_h)\CROSS \xHone_0(\Omega)$:
\begin{equation*}
a_h\left((\bE,p),(\bF+\bF_h,q)\right) = \left(\eps\bg,\bF+\bF_h\right)_\Omega 
+ c_\alpha h^{2(1-\alpha)}\left(\eps\bg,\GRAD q\right)_\Omega.
\end{equation*}
\end{lem}
\begin{proof}
Let us first observe that 
\begin{multline*}
a_h\left((\bE,p),(\bF+\bF_h,q)\right) = (\kappa\ROT \bE,\ROT (\bF+\bF_h))_{\Omega_\Sigma}
+ (\kappa\ROT\bE,\jump{\bF_h\CROSS\bn})_{\Sigma\cup \front} \\
+ (\eps\GRAD p,\bF+\bF_h)_\Omega + c_\alpha h^{2(1-\alpha)} (\eps\GRAD p, \GRAD q)_\Omega,
\end{multline*}
where all the terms make sense owing to the extension properties of
$a_h$ stated in Proposition~\ref{Prop:extension_ah}.  We now test
\eqref{bvp} with $\bF+\bF_h\in (\bZ^s+\bX_h)$,
\[
(\ROT(\kappa\ROT \bE),\bF)_\Omega + \sum_{i=1}^N (\ROT(\kappa\ROT \bE),\bF_h)_{\Omega_i} 
+ (\eps\GRAD p, \bF+\bF_h)_\Omega = (\eps \bg,\bF+\bF_h)_\Omega,
\]
and we perform the integration by parts over $\Omega$ when the test
function is $\bF$ and over each sub-domain when the test function is
$\bF_h$,
\begin{multline*}
(\kappa\ROT \bE,\ROT \bF)_\Omega + \sum_{i=1}^N (\kappa\ROT \bE,\ROT\bF_h)_{\Omega_i}
+ ( \kappa\ROT \bE,\jump{\bF_h\CROSS\bn})_{\Sigma\cup \front} + (\eps\GRAD p, \bF+\bF_h)_\Omega = (\eps \bg,\bF+\bF_h)_\Omega.
\end{multline*}
Note that the term $( \kappa\ROT
\bE,\jump{\bF_h\CROSS\bn})_{\Sigma\cup \front}$ is meaningful owing to
Lemma~\ref{l:cont_IP} and $\bE$ being a member of $\bZ^s$. This
implies that
\[
a_h\left((\bE,p),(\bF+\bF_h,q)\right) = (\eps \bg,\bF+\bF_h)_\Omega 
+ c_\alpha h^{2(1-\alpha)} (\eps\GRAD p, \GRAD q)_\Omega.
\]
Upon testing again \eqref{bvp} with $\GRAD q$, $q\in \xHone_0(\Omega)$, we infer
that $(\eps\GRAD p, \GRAD q)_\Omega = (\eps\bg, \GRAD q)_\Omega$,
which in turn implies the desired result.
\end{proof}

\subsection{Well posedness of the discrete formulation}
We discuss in this section the existence and uniqueness of a solution
$(\bE_h,p_h)$ to \eqref{e:discrete_formulation}.  This issue is
addressed by equipping $\bX_h \CROSS M_h$ with the following discrete norm:
\begin{equation}\label{discrete_norm} 
\begin{aligned}
  \norme[h]{\bF_h,q_h}^2 := &
  \norme[\bL^2(\Omega_\Sigma)]{\kappa^{\frac12}{\ROT\bF_h}}^2 + \gamma
  h^{-1} \norme[\xLtwo(\Sigma \cup \front)]{\mean{\kappa}^{\frac12} \jump{
      \bF_h \times \bn}}^2 \\ & + c_\alpha \Big(
  h^{2\alpha}\norme[\xLtwo(\Omega_\Sigma)]{\DIV(\eps\bF_h)}^2 +
  h^{2(1-\alpha)}\norme[\bL^2(\Omega)]{\eps^{\frac12}\GRAD q_h}^2\\ &
  +
  h^{(2\alpha-1)}\norme[\bL^2(\Sigma)]{\jump{\eps\bF_h\SCAL\bn}}^2\Big),
\end{aligned}
\end{equation}
by proving a coercivity property, uniform in $h$, and by establishing
some continuity estimates for the bilinear form $a_h(.,.)$.
\bal{Notice that we do not include the $L^2$-norm in the discrete
  norm since this quantity is better handled by a duality argument.
  We postpone this discussion to Section~\ref{sec:L2}.}

We first establish the coercivity of $a_h$.
\begin{prop}[Coercivity]
If $\theta\in \{0,1\}$, there exists $\gamma_0>0$ and $c(\gamma_0)>0$, 
uniform with respect to $h$, so that the following holds 
for all $\gamma\ge \gamma_0$ and for any $0\leq \alpha \leq 1$:
\begin{equation}
a_h((\bE_h,p_h),(\bE_h,p_h)) \ge c(\gamma_0)  \norme[h]{\bE_h,p_h}^2,\qquad
\forall (\bE_h,p_h)\in \bX_h\CROSS M_h,
\end{equation}
and this inequality holds for all $\gamma>0$ with $c(\gamma_0)=1$ if $\theta=-1$. 
\end{prop}

\begin{proof}
We first observe that
\[
a_h((\bE_h,p_h),(\bE_h,p_h)) = \norme[h]{\bE_h,p_h}^2 + (1+\theta)
\left( \mean{\kappa\ROT\bE_h},\jump{\bE_h\CROSS \bn}
\right)_{\Sigma\cup \front}.
\]
The conclusion is evident if $\theta=-1$. Otherwise we have to control
the term $\left( \mean{\kappa\ROT\bE_h},\jump{\bE_h\CROSS \bn}
\right)_{\Sigma\cup \front}$.
Invoking \bal{the inverse trace inequality \eqref{Inv_ineq3} and the inequality $ab \le \frac14 a^2 + b^2$, we 
deduce that there exists a constant $c_0$ only depending on the trace inequality constant and the ratio $\kappa_{\max}/\kappa_{\min}$  such that}
\[
\big|\left( \mean{\kappa\ROT\bF_h},\jump{\bF_h\CROSS \bn}
  \right)_{\Sigma\cup \front} \big| \leq \frac14
  \norme[\xLtwo(\Omega_\Sigma)]{\kappa^{\frac12}\ROT \bF_h}^2 +c_0 h^{-1} \norme[\xLtwo(\Sigma\cup
  \front)]{ \mean{\kappa}^{\frac12}\jump{ \bF_h \CROSS \bn}}^2.
\]
Hence, if $\gamma \ge \gamma_0:=4 c_0$, we infer that the
following holds:
\begin{equation}\label{r:positive}
  a_h\left((\bE_h,p_h),(\bE_h,p_h)\right) 
  \geq  \frac12
  \norme[h]{\bE_h,p_h}^2\geq 0.
\end{equation}
This completes the proof.
\end{proof}

We now establish the uniform boundedness of the bilinear form $a_h$.
\begin{prop}[Continuity]\label{p:cont}
  For any $s\in\left(0,\frac12\right)$, there is $c>0$, uniform in $h$
  such that the following holds for any $0\leq \alpha \leq 1$ and for
  every $(\bE,p)\in \bZ^{s} \CROSS \xHone_0(\Omega)$ and $(\bG_h,d_h),
  (\bF_h,q_h) \in\bX_h\CROSS M_h$:
\begin{align}
\label{cont}
c\;\frac{\bal{|a_h\left((\bE-\bG_h,p-d_h),(\bF_h,q_h)\right)|}}{\norme[h]{\bF_h,q_h}}
& \le \ba{\norme[h]{\bE-\bG_h,p-d_h}} + h^{\alpha-1}\norme[\bL^2(\Omega)]{\bE-\bG_h} \nonumber
\\&\hspace{-3cm}
+h^{{s}}\norme[\bH^{{s}}(\calT_h)]{\kappa \ROT (\bE-\bG_h)}
+ h \norme[\bL^2(\calT_h)]{\ROT \kappa \ROT (\bE-\bG_h)}\\
&\hspace{-3cm}+h^{-\alpha}\norme[\xLtwo(\Omega)]{p-d_h} +
h^{(\frac12-\alpha)}\norme[\xLtwo(\Sigma)]{p-d_h}.\nonumber
\end{align}
\end{prop}
\begin{proof}
Upon applying the Cauchy-Schwarz inequality we obtain
\begin{align*}
&\left(\kappa {\ROT (\bE-\bG_h)},{\ROT\bF_h}\right)_{\Omega_\Sigma} 
+\gamma h^{-1} \left( \mean{\kappa} \jump{(\bE-\bG_h)\CROSS \bn},
  \jump{\bF_h\CROSS \bn}\right)_{\Sigma\cup\front} \\
&\qquad + c_\alpha \Big(h^{2\alpha}\left(\DIV(\eps
  (\bE-\bG_h)),\DIV(\eps\bF_h)\right)_{\Omega_\Sigma} +
h^{2(1-\alpha)}\left(\eps\GRAD (p-d_h),\GRAD q_h\right)_\Omega\\
&\qquad
+h^{(2\alpha-1)}\left(\jump{\eps(\bE-\bG_h)\SCAL\bn},\jump{\eps\bF_h\SCAL\bn}\right)_\Sigma\Big)
\\
&\qquad \leq  \norme[h]{\bF_h,q_h} \norme[h]{\bE-\bG_h,p-d_h}. 
\end{align*}
We now bound separately the remaining terms appearing in the definition
\eqref{def_a_h} of $a_h(.,.)$:
\begin{align*}
\bal{|- \left(\eps (\bE-\bG_h),\GRAD q_h\right)_\Omega|} &\leq 
\bal{\|\eps\|_{L^\infty(\Omega)} }
h^{\alpha-1} \norme[\bL^2(\Omega)]{\nabla q_h}h^{1-\alpha}\norme[\bL^2(\Omega)]{\bE-\bG_h},\\
\left(\eps\GRAD( p-d_h),\bF_h\right)_\Omega  & \leq 
 h^\alpha \norme[\xLtwo(\Omega_\Sigma)]{\DIV(\eps
  \bF_h)}h^{-\alpha}\norme[\xLtwo(\Omega)]{p-d_h}\\
& +h^{(\alpha-\frac12)}
\norme[\bL^2(\Sigma)]{\jump{ \eps \bF_h\cdot \bn}}
h^{(\frac12-\alpha)}
\norme[\xLtwo(\Sigma)]{p-d_h},
\end{align*}
where we used an integration by parts for the second estimate. We are
now left with the consistency terms
\begin{equation}\label{mixed_term}
\left( \mean{\kappa\ROT(\bE-\bG_h)}, \jump{\bF_h\CROSS \bn} \right)_{\Sigma\cup\front}  
  + \theta \left( \mean{\kappa\ROT\bF_h}, \jump{(\bE-\bG_h)\CROSS \bn}\right)_{\Sigma\cup\front}.
\end{equation}
For the first term in \eqref{mixed_term}, we apply Lemma
\ref{l:cont_IP} with $\bv = \jump{\bF_h \CROSS \bn}$, which is a
polynomial of degree $\ell-1$, and $\bphi = \mean{\kappa\ROT
  (\bE-\bG_h)}$. Then for any $F \in \calF_h$, we infer
that
\begin{equation*}
\begin{aligned}
&\bal{|\left( \mean{\kappa\ROT(\bE-\bG_h)}, \jump{\bF_h\CROSS \bn} \right)_{F}|}
\leq c h^{-\frac12} \norme[\bL^2(F)]{\jump{\bF_h \CROSS \bn}}\\
&\qquad \times \sum_{i=1}^2\Big(
  h^{s}\norme[\bH^{s}(K_i)]{\kappa \ROT (\bE-\bG_h)}+ h
  \norme[\bL^2(K_i)]{\ROT \kappa \ROT (\bE-\bG_h)}\\
& \qquad \qquad \qquad + \norme[\bL^2(K_i)]{\kappa
    \ROT (\bE-\bG_h)}  \Big),
\end{aligned}
\end{equation*}
where $K_1,K_2 \in \calT_h$ such that $F=\overline{K_1} \cap \overline{K_2}$.
Hence, summing over all the faces we arrive at
\begin{align*}
 \bal{| ( \mean{\kappa\ROT(\bE-\bG_h)}}&, \bal{\jump{\bF_h\CROSS \bn})_{\Sigma \cup \front}|}  
\leq c \, h^{-\frac12} \norme[\bL^2(\Sigma
  \cup \front)]{\jump{\bF_h \CROSS \bn}}
 \Big( h^{s}\norme[\bH^{s}(\calT_h)]{\kappa \ROT
    (\bE-\bG_h)} \\ & + h
  \norme[\bL^2(\calT_h)]{\ROT \kappa \ROT (\bE-\bG_h)}
 + \norme[\bL^2(\calT_h)]{\kappa \ROT
    (\bE-\bG_h)} \Big).
\end{align*}
For the second term in \eqref{mixed_term} we notice that
$\jump{(\bE-\bG_h)\CROSS \bn} = -\jump{\bG_h\CROSS \bn}$ owing to the
regularity of $\bE$. Then by using
Lemma~\ref{l:cont_IP} again, we arrive at 
\begin{equation*}
\begin{aligned}
  \bal{|\left( \mean{\kappa\ROT\bF_h}, \jump{(\bE-\bG_h)\CROSS \bn}
  \right)_{\Sigma \cup \front} |}& \leq c h^{-\frac12}
  \norme[\xLtwo(\Sigma\cup\front)]{\jump{\bG_h\CROSS \bn}}
  \norme[\xLtwo(\calT_h)]{\kappa \ROT \bF_h}\\
&\leq c h^{-\frac12}
  \norme[\xLtwo(\Sigma\cup\front)]{\jump{(\bE-\bG_h)\CROSS \bn}}
  \norme[\xLtwo(\calT_h)]{\kappa \ROT \bF_h},
\end{aligned}
\end{equation*}
where we used the inverse inequalities \eqref{Inv_ineq1}, \eqref{Inv_ineq2}. 
The desired result is obtained by gathering the above estimates.
\end{proof}

The following result will be instrumental to apply the Nitsche-Aubin
duality argument and derive a convergence result in $\bL^2(\Omega)$.
\begin{prop}[Adjoint continuity]\label{prop:cont_bis}
  For any $s\in\left(0,\frac12\right)$, there is $c>0$, uniform in $h$
  such that for any $0\leq \alpha \leq 1$, the following holds for
  every $(\bE,p),(\bF,q)\in \bZ^{s} \CROSS \xHone_0(\Omega)$,
  $\bF_h\in\bY_h$, $q_h\in M_h$ and $(\bG_h,d_h)\in\bX_h\CROSS M_h$:
\begin{align}\label{cont_bis}
  c\;\frac{\bal{|a_h\left((\bE-\bG_h,p-d_h),(\bF-\bF_h,q-q_h)\right)|}}{\norme[h]{\bE-\bG_h,p-d_h}}
  &\le \norme[h]{\bF-\bF_h,q-q_h} + h^{\alpha-1}\norme[\bL^2(\Omega)]{\bF-\bF_h} \nonumber\\
  &+
  h^{s}\norme[\xHn{s}(\calT_h)]{\kappa \ROT (\bF-\bF_h)}\nonumber\\
  &+ h \norme[\xLtwo(\calT_h)]{\ROT \kappa \ROT (\bF-\bF_h)}\\
  &+h^{-\alpha}\norme[\xLtwo(\Omega)]{q-q_h} +
  h^{(\frac12-\alpha)}\norme[\xLtwo(\Sigma)]{q-q_h}.\nonumber
\end{align}

\end{prop}
\begin{proof}
  The proof proceeds similarly as in the proof of
  Proposition~\ref{p:cont}. The only difference here is that we have $
  \left(\mean{\ROT(\bE-\bG_h)},\jump{(\bF-\bF_h)\CROSS\bn}\right)_{\Sigma\cup\front}
  = 0$, owing to the assumption that $\bF_h\in \bY_h \subset
  \bX_h\cap \bH^1_0(\Omega)$. This identity makes the analysis of the
  consistency term \eqref{mixed_term} tractable.
\end{proof}

\section{Convergence analysis for the boundary value problem}
\label{Sec:Convergence_analysis_BVP}
In the first part of this section, we prove two convergence results
for the discrete problem \eqref{e:discrete_formulation} using the
discrete norm $\|\cdot\|_h$, one assuming minimal regularity and the
other assuming full smoothness. In the second part of the section we
use a Nitsche-Aubin duality argument to establish convergence in
$\bL^2(\Omega)$. The performance of the method is numerically
illustrated at the end of the section.

\subsection{Convergence in the discrete norm.}
We assume first that the solution to the boundary value problem
\eqref{bvp} has minimal regularity properties, and we start with the
Galerkin orthogonality.
\begin{lem}[Galerkin Orthogonality] Assume \eqref{Hyp:eps_mu}, then
  the Galerkin orthogonality holds, \ie let $(\bE,p)$ be the solution
  of \eqref{bvp} and $(\bE_h,p_h)$ be the solution of
  \eqref{e:discrete_formulation}, then for any $(\bF_h,q_h)\in \bX_h
  \CROSS M_h$
\begin{equation}\label{galerkin_ortho}
a_h\left((\bE-\bE_h,p-p_h),(\bF_h,q_h)\right) = 0.
\end{equation}
\end{lem}
\begin{proof} 
  This is a direct consequence of Lemma~\ref{Lem:consistency} and
  formulation \eqref{e:discrete_formulation}.
\end{proof}

\begin{thm}\label{thm:cvg_norme_h_min_reg}
  Let $\bg\in\bL^2(\Omega)$ and $\tau\in
  (0,\min(\tau_\eps,\tau_\kappa))$ where $\tau_\eps$ and $\tau_\kappa$
  are defined in Theorem~\ref{Thm:Hs_stability_bvp}. Let $(\bE,p)$ and
  $(\bE_h,p_h)$ be the solution of~\eqref{bvp} and
  \eqref{e:discrete_formulation}, respectively. Then, for any
  $\alpha\in\left(\frac{\ell(1-\tau)}{\ell-\tau},1\right]$, there
  exists $c>0$, uniform in $h$, such that
\begin{equation}\label{eq:cvg_norme_h_min_reg}
\|\bE-\bE_h,p-p_h\|_h \le ch^r\|\bg\|_{\bL^2(\Omega)},
\end{equation}
where $r=\alpha-1+\tau\left(1-\frac{\alpha}{\ell}\right)$ if
$\DIV(\eps\bg)=0$ and
$r=\min\left(1-\alpha,\alpha-1+\tau\left(1-\frac{\alpha}{\ell}\right)\right)$
otherwise.
\end{thm}

\begin{proof}
  We first recall that, owing to Theorem~\ref{Thm:Hs_stability_bvp},
  we have $\bE\in\bH^{\tau}(\Omega)\cap\Hzrcurl[\Omega][\tau]$,
  together with the estimates
\[
\|\bE\|_{\bH^\tau(\Omega)} + \|\ROT\bE\|_{\bH^{\tau}(\Omega)} +
\|\ROT(\kappa\ROT\bE)\|_{\bL^2(\Omega)} +\|\GRAD p\|_{\bL^2(\Omega)}
\le c\, \|\bg\|_{\bL^2(\Omega)}.
\]
We establish~\eqref{eq:cvg_norme_h_min_reg} by using the triangular
inequality
\begin{align*}
  \|\bE-\bE_h,p-p_h\|_h & \le \|\bE-\calK_\delta\bE,0\|_h 
+\|\calK_\delta\bE-\calC_h^g\calK_\delta\bE,p-\calC^p_hp\|_h \\
  &+\|\calC_h\calK_\delta\bE-\bE_h,\calC^p_hp-p_h\|_h,
\end{align*}
for some $\delta>0$ to be defined later, and by bounding from above
the three terms separately.  

Using the definition of $\|\cdot\|_h$ together with the approximation
properties of $\calK_\delta$,
\cf~\eqref{approx_delta}-\eqref{approx_delta_rot}-\eqref{approx_exp},
we conclude that
\begin{align*}
  \|\bE-\calK_\delta\bE,0\|_h &\le
  c\,\left(\delta^{\tau}\|\ROT\bE\|_{\bH^{\tau}(\Omega)}
  +h^{\alpha}\delta^{\tau-1}\|\bE\|_{\bH^\tau(\Omega)}
  +h^{\alpha-\frac12}\|\calK_\delta\bE\|_{\bL^2(\Sigma)}\right).
\end{align*}
Note that the estimate \eqref{approx_delta_rot} is critical to obtain
a bound that depends only on $\|\ROT\bE\|_{\bH^{\tau}(\Omega)}$
instead of $\|\bE\|_{\bH^{1+\tau}(\Omega)}$.  To estimate the last
term in the above inequality, we apply~\eqref{eq:L2Gamma_2} with
$\Theta = \frac{1-2\tau}{2(1-\tau)}$,
\begin{align*}
  h^{\alpha-\frac12}\|\calK_\delta\bE\|_{\bL^2(\Sigma)}&
  \le ch^{\alpha-\frac12}\|\calK_\delta\bE\|_{\bH^\tau(\Omega)}^{1-\Theta}\|\calK_\delta\bE\|_{\bH^1(\Omega)}^{\Theta} \\
  &\le ch^{\alpha-\frac12}\delta^{\Theta(\tau-1)}\|\bE\|_{\bH^\tau(\Omega)}
  \le
  ch^{\alpha-\frac12}\delta^{\tau-\frac12}\|\bE\|_{\bH^\tau(\Omega)}.
\end{align*}
Finally, we arrive at
\begin{equation}\label{eq:cvg_norm_h_min_reg_1}
  \|\bE-\calK_\delta\bE,0\|_h \le c\left(\delta^\tau+h^{\alpha}\delta^{\tau-1}
    +h^{\alpha-\frac12}\delta^{\tau-\frac12}\right)\|\bg\|_{\bL^2(\Omega)}.
\end{equation}

Let us now turn our attention to
$\|\calK_\delta\bE-\calC_h^g\calK_\delta\bE,p-\calC_h^p p\|_h$. Owing
to the definition of $\calC_h^g$ and the regularity of
$\calK_\delta\bE$, we have
$\calC_h^g\calK_\delta\bE\in\bH^1_0(\Omega)\subset \Hzcurl$, so that
we only have four terms to bound (the jumps of
$\calC_h^g\calK_\delta\bE$ across the mesh interfaces and the tangent
trace on $\front$ are zero, \cf Remark~\ref{rem:chf_smooth}). Using
the properties of $\calK_\delta$ and $\calC_h^g$ together with
~\eqref{eq:L2Gamma_1} we deduce that
\begin{align*}
  \|\kappa^{\frac12}\ROT(\calK_\delta\bE-\calC_h^g\calK_\delta\bE)\|_{\bL^2(\Omega)}
  &\le c\,h^{\ell-1}\|\calK_\delta\bE\|_{\bH^{\ell}(\Omega)}
  \le ch^{\ell-1}\delta^{\tau-\ell}\|\bE\|_{\bH^{\tau}(\Omega)}, \\
  h^{\alpha}\|\DIV\left(\eps(\calK_\delta\bE
    -\calC_h^g\calK_\delta\bE)\right)\|_{\xLtwo(\Omega_\Sigma)} &\le
  c\,h^{\alpha+\ell-1}\|\calK_\delta\bE\|_{\bH^{\ell}(\Omega_\Sigma)}
  \le ch^{\alpha+\ell-1}\delta^{\tau-\ell}\|\bE\|_{\bH^{\tau}(\Omega)}, \\
  h^{1-\alpha}\|\eps^{\frac12}\GRAD(p-\calC_h^p p)\|_{\bL^2(\Omega)}
  &\le c\,h^{1-\alpha}\|p\|_{\xHone_0(\Omega)}, \\
  h^{\alpha-\frac12}\|\jump{\eps(\calK_\delta\bE
    -\calC_h^g\calK_\delta\bE)\SCAL\bn}\|_{\bL^2(\Sigma)} &\le
  c\,h^{\alpha-\frac12} \|\calK_\delta\bE
  -\calC_h^g\calK_\delta\bE\|_{\bL^2(\Sigma)} \\
  &\le c\,h^{\alpha-\frac12}\|\calK_\delta\bE
  -\calC_h^g\calK_\delta\bE\|_{\bL^2(\Omega)}^{1-\frac1{2\alpha}}\|\calK_\delta\bE-\calC_h^g\calK_\delta\bE\|_{\bH^\alpha(\Omega)}^{\frac1{2\alpha}} \\
  &\le c\,h^{\alpha-\frac12} h^{\ell\left(1-\frac1{2\alpha}\right)}h^{(\ell-\alpha)\frac1{2\alpha}}\|\calK_\delta\bE\|_{\bH^{\ell}(\Omega)} \\
  &\le
  c\,h^{\alpha+\ell-1}\delta^{\tau-\ell}\|\bE\|_{\bH^\tau(\Omega)}.
\end{align*}
When combining the above estimates, we obtain
\begin{equation}\label{eq:cvg_norm_h_min_reg_2}
\|\calK_\delta\bE-\calC_h^g\calK_\delta\bE,p-\calC_h^p p\|_h
\le c\left(h^{\ell-1}\delta^{\tau-\ell}+\xi h^{1-\alpha}\right)\|\bg\|_{\bL^2(\Omega)},
\end{equation}
where $\xi=0$ if $\DIV{(\eps\bg)}=0$ and $\xi=1$ otherwise (note that $p=0$
when $\DIV{(\eps\bg)}=0$).

The last term, $\|\calC_h^g\calK_\delta\bE-\bE_h,\calC_h^p p-p_h\|_h$, is a
little more subtle to handle. We start from the coercivity of
$a_h$, \eqref{r:positive}, and use both the Galerkin
orthogonality~\eqref{galerkin_ortho} and the continuity of
$a_h$, \eqref{cont}, with $s=1-\alpha$ to get the following estimate:
\begin{align*}
  \|\calC_h^g\calK_\delta\bE&-\bE_h,\calC_h^p p-p_h\|_h \\
  & \le c\, \frac{a_h\left((\calC_h^g\calK_\delta\bE-\bE_h,\calC_h^p
      p-p_h),
      (\calC_h^g\calK_\delta\bE-\bE_h,\calC_h^p p-p_h)\right)}{\|\calC_h^g\calK_\delta\bE-\bE_h,\calC_h^p p-p_h\|_h} \\
  & \le c\, \frac{a_h\left((\calC_h^g\calK_\delta\bE-\bE,\calC_h^p
      p-p),
      (\calC_h^g\calK_\delta\bE-\bE_h,\calC_h^p p-p_h)\right)}{\|\calC_h^g\calK_\delta\bE-\bE_h,\calC_h^p p-p_h\|_h} \\
  & \le c\, \big(\|\calC_h^g\calK_\delta\bE-\bE,\calC_h^p p-p\|_h
  + h^{\alpha-1}\|\bE-\calC_h^g\calK_\delta\bE\|_{\bL^2(\Omega)} \\
  &\quad +
  h^{1-\alpha}\|\kappa\ROT(\bE-\calC_h^g\calK_\delta\bE)\|_{\bH^{1-\alpha}(\Omega)}
  + h^{-\alpha}\|p-\calC_h^p  p\|_{\xLtwo(\Omega)} \\
  &\quad +
  h\|\ROT\kappa\ROT(\bE-\calC_h^g\calK_\delta\bE)\|_{\bL^2(\calT_h)}+
  h^{\frac12-\alpha}\|p-\calC_h^p p\|_{\xLtwo(\Sigma)}\big).
\end{align*}
We now handle each term in the right hand side separately. Using the
triangle inequality $\|\calC_h^g\calK_\delta\bE-\bE,\calC_h^p p-p\|_h \le
\|\calC_h^g\calK_\delta\bE-\calK_\delta\bE,\calC_h^p p-p\|_h +
\|\calK_\delta\bE-\bE,0\|_h $ and the estimates 
\eqref{eq:cvg_norm_h_min_reg_1}-\eqref{eq:cvg_norm_h_min_reg_2},
we obtain 
\[
\|\calC_h^g\calK_\delta\bE-\bE,\calC_h^p p-p\|_h \le
c\left(\delta^\tau+h^{\alpha}\delta^{\tau-1}
  +h^{\alpha-\frac12}\delta^{\tau-\frac12}
  +h^{\ell-1}\delta^{\tau-\ell}+\xi h^{1-\alpha}\right)\|\bg\|_{\bL^2(\Omega)}.
\]
Similarly, we obtain
\begin{align*}
  h^{\alpha-1}\|\bE-\calC_h^g\calK_\delta\bE\|_{\bL^2(\Omega)}
  &\le c\left(h^{\alpha-1}\delta^\tau+h^{\alpha+\ell-1}\delta^{\tau-\ell}\right)\|\bg\|_{\bL^2(\Omega)}.
\end{align*}
\bal{Now using that $\kappa\ROT\bE\in \bH^\tau(\Omega)$ and $1-\alpha\le\tau$,
owing to the assumption
$\alpha\in\left(\frac{\ell(1-\tau)}{\ell-\tau},1\right]$, we infer that}
\begin{align*}
  h^{1-\alpha}\|\kappa\ROT(\bE-\calC_h^g\calK_\delta\bE)\|_{\bH^{1-\alpha}(\Omega)}
  &\le c\left(h^{1-\alpha}\delta^{\tau+\alpha-1} +
    h^{\ell-1}\delta^{\tau-\ell}\right)\|\bg\|_{\bL^2(\Omega)}.
\end{align*}
For the last term involving $\bE$ we use the commuting property
$\bar\delta\ROT\calK_\delta\bE = \calK_\delta\ROT\bE$, see
\eqref{eq:commute_K}, to derive
\begin{align*}
h\|\ROT(\kappa\ROT(\bE&-\calC_h^g\calK_\delta\bE))\|_{\bL^2(\calT_h)} 
\le h\|\ROT(\kappa\ROT\bE)\|_{\bL^2(\calT_h)}
+h\|\ROT(\kappa\ROT\calK_\delta\bE)\|_{\bL^2(\calT_h)} \\
&+ h\|\ROT(\kappa\ROT(\calK_\delta\bE-\calC_h^g\calK_\delta\bE))\|_{\bL^2(\calT_h)} \\
&\le c\left(h\|\bg\|_{\bL^2(\Omega)} + h\|\ROT\calK_\delta\bE\|_{\bH^1(\Omega)} 
+ h^{\ell-1}\|\calK_\delta\bE\|_{\bH^{\ell}(\Omega)}\right) \\
&\le c\left(h\|\bg\|_{\bL^2(\Omega)} + h\|\calK_\delta\ROT\bE\|_{\bH^1(\Omega)} 
+ h^{\ell-1}\delta^{\tau-\ell}\|\bE\|_{\bH^{\tau}(\Omega)}\right) \\
&\le c\left(h + h\delta^{\tau-1} + h^{\ell-1}\delta^{\tau-\ell} \right) \|\bg\|_{\bL^2(\Omega)}.
\end{align*}
For the remaining terms involving $p$, we use~\eqref{eq:L2Gamma_1}
together with the approximation properties of $\calC_h^p$:
\begin{align*}
  h^{-\alpha}\|p-\calC_h^p p\|_{\xLtwo(\Omega)} &\le
  c\, h^{1-\alpha}\|p\|_{\xHone_0(\Omega)}
  \le c\xi h^{1-\alpha}\|\bg\|_{\bL^2(\Omega)}, \\
  h^{\frac12-\alpha}\|p-\calC_h^p p\|_{\xLtwo(\Sigma)} &\le
  h^{\frac12-\alpha}\|p-\calC_h^p
  p\|_{\xLtwo(\Omega)}^{1-\frac1{2\alpha}}
  \|p-\calC_h^p p\|_{\xHn{\alpha}(\Omega)}^{\frac1{2\alpha}} \\
  &  \le c\,h^{\frac12-\alpha}
  h^{1-\frac1{2\alpha}}h^{(1-\alpha)\frac1{2\alpha}}\|p\|_{\xHone_0(\Omega)}
  \le c \xi h^{1-\alpha}\|\bg\|_{\bL^2(\Omega)}.
\end{align*}
Gathering all the above estimates together with~\eqref{eq:cvg_norm_h_min_reg_1} 
and~\eqref{eq:cvg_norm_h_min_reg_2}, we finally obtain
\begin{equation}\label{eq:cvg_norm_h_min_reg_3}
\begin{aligned}
\|\bE-\bE_h,p-p_h\|_h \le &c\big(\delta^{\tau} + \xi h^{1-\alpha} 
+ h + h\delta^{\tau-1} + h^{\ell-1}\delta^{\tau-\ell}+h^{\alpha-1}\delta^\tau\\ 
&+h^{1-\alpha}\delta^{\tau+\alpha-1}+h^{\alpha}\delta^{\tau-1} 
+ h^{\alpha-\frac12}\delta^{\tau-\frac12}+\bal{h^{\alpha+\ell-1}\delta^{\tau-\ell}}\big)\|\bg\|_{\bL^2(\Omega)}.
\end{aligned}
\end{equation}
We want to use $\delta=h^\beta$ for some $\beta\in(0,1)$, \ie $\delta
h^{-1}\to +\infty$ as $h\to 0$.  Once the negligible terms are removed
in~\eqref{eq:cvg_norm_h_min_reg_3}, we derive the following estimate:
\[
\|\bE-\bE_h,p-p_h\|_h \le c\big(h^{\alpha-1}\delta^\tau 
+ \xi h^{1-\alpha} + h^{\ell-1}\delta^{\tau-\ell}\big)\|\bg\|_{\bL^2(\Omega)}.
\]
Using $\delta = h^{1-\frac{\alpha}{\ell}}$ implies that
$h^{\alpha-1}\delta^\tau = h^{\ell-1}\delta^{\tau-\ell}$ and we arrive
at
\[
\|\bE-\bE_h,p-p_h\|_h \le c(h^{\alpha-1+\tau\left(1-\frac{\alpha}{\ell}\right)}
+\xi h^{1-\alpha})\|\bg\|_{\bL^2(\Omega)},
\]
which leads to~\eqref{eq:cvg_norme_h_min_reg} with
$r:=\min\left(1-\alpha,\alpha-1+\tau\left(1-\frac{\alpha}{\ell}\right)\right)$
if $\DIV(\eps\bg)\not=0$ and $r=\alpha-1+\tau\left(1-\frac{\alpha}{\ell}\right)$
  otherwise.  Note that the assumed lower bound on $\alpha$ ensures
  that we have a convergence result as $h\to 0$.
\end{proof}

\begin{rem}[$\alpha=1$] \label{rem:s_less_14} {Note that the best
  choice for $\alpha$ when $\DIV(\eps\bg)=0$ is $\alpha=1$; the
  convergence rate is then $\tau\left(1-\frac{1}{\ell}\right)$ and it
  approaches the optimal rate $\tau$ as $\ell$ increases.
  When $\DIV(\eps\bg)\ne 0$, the best choice for $\alpha$ is such
  that $1-\alpha =
  \alpha-1+\tau\left(1-\frac{\alpha}{\ell}\right)$. This choice gives
  the following convergence rate $\frac{\tau}{2}(1-\frac{1}{\ell}) < r
  = {\tau}\frac{\ell-1}{2\ell-\tau} < \frac{\tau}{2}$.}
\end{rem}

We now derive a convergence estimate assuming that the solution of
\eqref{bvp} is smooth. In the next theorem we allow the parameter
$\alpha$ to be any number in the interval $[0,1]$.
\begin{thm}\label{thm:cvg_norme_h_smooth}
  Let $\bg\in\bL^2(\Omega)$ and let $(\bE,p)$ and $(\bE_h,p_h)$ be the
  solution of~\eqref{bvp} and \eqref{e:discrete_formulation},
  respectively.   Assume moreover that
  $\bE\in\bH^{k+1}(\Omega_\Sigma)$ and $p\in \xHn{k+\alpha}(\Omega_\Sigma)$ for some
  $0<k\le \ell-1$. Then there exists $c>0$, uniform in $h$, such that
\begin{equation}\label{eq:cvg_norme_h_smooth}
  \|\bE-\bE_h,p-p_h\|_h \le c\,h^k\left(\|\bg\|_{\bL^2(\Omega)}
    +\|\bE\|_{\bH^{k+1}(\Omega_\Sigma)}+\|p\|_{\xHn{k+\alpha}(\Omega_\Sigma)}\right).
\end{equation}
\end{thm}
\begin{proof}
  The proof is similar to that of
  Theorem~\ref{thm:cvg_norme_h_min_reg}. We start from the triangular
  inequality
\[
\|\bE-\bE_h,p-p_h\|_h \le \|\bE-\calC_h^l\bE,p-\calC_h^p p\|_{h}
+ \|\calC_h^l\bE-\bE_h,\calC_h^p p-p_h\|_{h},
\]
We bound the two terms in the right hand side separately. For the
first one, we use the local approximation properties of the operators
$\calC_h^l$ and $\calC_h^p$ to derive
\begin{align*}
  \|\bE-\calC_h^l\bE,p&-\calC_h^p p\|_{h}  \le c\,\Big(
  h^k\|\bE\|_{\bH^{k+1}(\Omega_\Sigma)} +
  h^{k+\alpha}\|\bE\|_{\bH^{k+1}(\Omega_\Sigma)}
  +  h^{-\frac12}\|\bE-\calC_h^l\bE\|_{\bL^2(\Sigma\cup\front)} \\
  & + h^{1-\alpha}h^{k+\alpha-1}\|p\|_{\xHn{k+\alpha}(\Omega_\Sigma)} +
  h^{\alpha-\frac12}\|\bE-\calC_h^l\bE\|_{\bL^2(\Sigma)}\Big).
\end{align*}
Using~\eqref{eq:L2Gamma_1} for any $\sigma\in\left(\frac12,1\right)$,
we have
\begin{align*}
  \|\bE-\calC_h^l\bE\|_{\bL^2(\Sigma\cup\front)} &\le c\,
  \|\bE-\calC_h^l\bE\|_{\bL^2(\Omega_\Sigma)}^{1-\frac1{2\sigma}}
\|\bE-\calC_h^l\bE\|_{\bH^\sigma(\Omega_\Sigma)}^{\frac1{2\sigma}}
  \le c\,h^{k+\frac12}\|\bE\|_{\bH^{k+1}(\Omega_\Sigma)}.
\end{align*}
As a result, we obtain
\begin{equation}\label{eq:cvg_norme_h_smooth_1}
  \|\bE-\calC_h^l\bE,p-\calC_h^p p\|_h 
  \le ch^k\left(\|\bE\|_{\bH^{k+1}(\Omega_\Sigma)}+\|p\|_{\xHn{k+\alpha}(\Omega_\Sigma)}\right).
\end{equation}
Now we turn our attention to $\|\calC_h^l\bE-\bE_h,\calC_h^p
p-p_h\|_{h}$. We use the coercivity of $a_h$, the
Galerkin orthogonality and the continuity  of $a_h$ (for any
$\sigma\in\left(0,\frac12\right)$) to get
\begin{align*}
  \|\calC_h^l\bE-\bE_h,\calC_h^p p-p_h\|_{h}
  &\le \ c\, \big(\|\bE-\calC_h^l\bE,p-\calC_h^p p\|_{h} + h^{\alpha-1}\|\bE-\calC_h^l\bE\|_{\bL^{2}(\Omega)} \\
  &+ h^{\sigma}\|\kappa\ROT(\bE-\calC_h^l\bE)\|_{\bH^\sigma(\calT_h)} \\
  &+ h\|\ROT\kappa\ROT(\bE-\calC_h^l\bE)\|_{\bL^2(\calT_h)} \\
  &+ h^{-\alpha}\|p-\calC_h^p p\|_{\xLtwo(\Omega)} +
  h^{\frac12-\alpha}\|p-\calC_h^p p\|_{\xLtwo(\Sigma)}\big).
\end{align*}
Using the approximation properties of $\calC_h^l$ together
with~\eqref{eq:cvg_norme_h_smooth_1}, we infer
\begin{align*}
  h^{\alpha-1}\|\bE-\calC_h^l\bE\|_{\bL^{2}(\Omega)}
  &\le ch^{k+\alpha}\|\bE\|_{\bH^{k+1}(\Omega_\Sigma)}, \\
  h^{\sigma}\|\kappa\ROT(\bE-\calC_h^l\bE)\|_{\bH^\sigma(\calT_h)}
  &\le ch^k\|\bE\|_{\bH^{k+1}(\Omega_\Sigma)}, \\
  h^{-\alpha}\|p-\calC_h^p p\|_{\xLtwo(\Omega)} &\le
  ch^k\|p\|_{\xHn{k+\alpha}(\Omega_\Sigma)}.
\end{align*}
For the last term involving $p$, we use~\eqref{eq:L2Gamma_1} for some
$\sigma\in\left(\frac12,1\right)$:
\begin{align*}
  h^{\frac12-\alpha}\|p-\calC_h^p p\|_{\xLtwo(\Sigma)} 
&\le ch^{\frac12-\alpha}\|p-\calC_h^p p\|_{\xLtwo(\Omega_\Sigma)}^{1-\frac1{2\sigma}}
\|p-\calC_h^p p\|_{\xHn{\sigma}(\Omega_\Sigma)}^{\frac1{2\sigma}} \\
  &\le
  ch^{\frac12-\alpha}h^{k+\alpha-\frac12}\|p\|_{\xHn{k+\alpha}(\Omega_\Sigma)}
  = ch^k\|p\|_{\xHn{k+\alpha}(\Omega_\Sigma)}.
\end{align*}
For the last term involving $\bE$, we distinguish two cases depending
whether $k<1$ or $k\ge 1$.  If $k<1$, we use an inverse inequality
together with the approximation properties of $\calC_h^l$ to deduce that
\begin{align*}
h\|\ROT\kappa\ROT(\bE-\calC_h^l\bE)\|_{\bL^2(\calT_h)} 
&\le h\|\ROT\kappa\ROT\bE\|_{\bL^2(\Omega)} + ch\|\calC_h^l\bE\|_{\bH^2(\calT_h)} \\
&\le h\|\bg\|_{\bL^2(\Omega)} + h^k\|\bE\|_{\bH^{k+1}(\Omega_\Sigma)}.
\end{align*}
If $k\ge 1$, we use the local approximation properties of $\calC_h^l$
to get
\begin{align*}
h\|\ROT\kappa\ROT(\bE-\calC_h^l\bE)\|_{\bL^2(\calT_h)} 
&\le ch\|\bE-\calC_h^l\bE\|_{\bH^2(\calT_h)} 
\le ch^k\|\bE\|_{\bH^{k+1}(\Omega_\Sigma)}.
\end{align*}
In both cases, we have:
\begin{align*}
h\|\ROT\kappa\ROT(\bE-\calC_h^l\bE)\|_{\bL^2(\calT_h)} 
&\le ch^k\left(\|\bE\|_{\bH^{k+1}(\Omega_\Sigma)}+\|\bg\|_{\bL^2(\Omega)}\right).
\end{align*}
Gathering all the above estimates and
using~\eqref{eq:cvg_norme_h_smooth_1} gives the desired
result~\eqref{eq:cvg_norme_h_smooth}.
\end{proof}

\begin{rem}
  Note that the error estimate \eqref{eq:cvg_norme_h_smooth} is
  optimal since it implies that $\|\ROT(\bE-\bE_h)\|_{\bL^2(\Omega_\Sigma)}
  \le c\, h^k$, which is the best that can be expected from
  piece-wise polynomial approximation of degree $k$.  Note also that
  there is no lower bound on $\alpha$ to get convergence when the
  solution of \eqref{bvp} is smooth, \ie any $\alpha$ in the range
  $[0,1]$ is acceptable.
\end{rem}

\subsection{Convergence in the $\bL^2$-norm.}\label{sec:L2}
Before proving that the discrete solution converges to the exact
solution in the $\bL^2$-norm, we prove a global version of
Lemma~\ref{l:cont_IP} that will be useful in the proof of
Theorem~\ref{thm:cvg_norme_l2}.
\begin{lem}\label{lem:approx_b_lambda}
  Let $s\in\left(0,\frac12\right)$. Then there exists $c>0$, uniform
  in $h$, such that the following holds, for any
  $\bpsi\in\Hcurl \cap\bH^s(\Omega)$ and any $\bF_h\in\bX_h$:
\begin{equation}\label{eq:approx_b_lambda}
  \left|\left(\bpsi,\jump{\bF_h\CROSS\bn}\right)_{\Sigma\cup\front}\right| 
  \le
  c\, h^{-\frac12}\|\jump{\bF_h\CROSS\bn}\|_{\bL^2(\Sigma\cup\front)}\;
\left(h^s\|\bpsi\|_{\bH^{s}(\Omega)}
    +h\|\ROT\bpsi\|_{\bL^{2}(\Omega)}\right).
\end{equation}
\end{lem}
\begin{proof}Let us consider $\bpsi\in\Hcurl \cap\bH^s(\Omega)$ and
  $\bF_h\in\bX_h$.  Notice that the left hand side is well defined
  owing to Lemma~\ref{l:cont_IP}.  We start from
\begin{align*}
  \big|\big(\bpsi&,
\jump{\bF_h\CROSS\bn}\big)_{\Sigma\cup\front}\big| 
  \le
\underbrace{\left|\left(\bpsi-\calK_\delta\bpsi,
\jump{\bF_h\CROSS\bn}\right)_{\Sigma\cup\front}\right|}_{:=I_1}
  +
  \underbrace{\left|\left(\calK_\delta\bpsi,
\jump{\bF_h\CROSS\bn}\right)_{\Sigma\cup\front}\right|}_{:=I_2},
\end{align*}
for some $\delta$ to be defined later. We handle the two terms $I_1$,
$I_2$ separately. For the first one, we apply Lemma~\ref{l:cont_IP}
with $\bv=\jump{\bF_h\CROSS\bn}$,
$\bphi=\bpsi-\calK_\delta\bpsi$ and $\sigma=s$, and we sum
over all the faces $F\in\Sigma\cup\front$. This leads to
\begin{align*}
I_1 &\le c\, h^{-\frac12}\|\jump{\bF_h\CROSS\bn}\|_{\bL^2(\Sigma\cup\front)}
\big( h^s\|\bpsi-\calK_\delta\bpsi\|_{\bH^s(\Omega_\Sigma)} \\
&\quad + h\|\ROT(\bpsi-\calK_\delta\bpsi)\|_{\bL^2(\Omega_\Sigma)} 
+ \|\bpsi-\calK_\delta\bpsi\|_{\bL^2(\Omega_\Sigma)}\big) \\
& \le c\, h^{-\frac12}\|\jump{\bF_h\CROSS\bn}\|_{\bL^2(\Sigma\cup\front)}
\big( h^s\|\bpsi-\calK_\delta\bpsi\|_{\bH^s(\Omega_\Sigma)} \\
&\quad + h\|\ROT\bpsi\|_{\bL^2(\Omega_\Sigma)}
+ h\|\ROT\calK_\delta\bpsi\|_{\bL^2(\Omega_\Sigma)} 
+ \|\bpsi-\calK_\delta\bpsi\|_{\bL^2(\Omega_\Sigma)}\big).
\end{align*}
Using the approximation properties of
$\calK_\delta$~\eqref{approx_delta} and \eqref{approx_exp}, we arrive
at
\begin{align*}
  I_1 & \le
  c\,h^{-\frac12}\|\jump{\bF_h\CROSS\bn}\|_{\bL^2(\Sigma\cup\front)}
  \big(h^s\|\bpsi\|_{\bH^s(\Omega_\Sigma)}\\
  &\quad + h\|\ROT\bpsi\|_{\bL^2(\Omega_\Sigma)}
  +\delta^s\|\bpsi\|_{\bH^s(\Omega_\Sigma)}
  +h\|\calK_\delta\bpsi\|_{\bH^1(\Omega_\Sigma)}\big) \\
  & \le c\,
  h^{-\frac12}\|\jump{\bF_h\CROSS\bn}\|_{\bL^2(\Sigma\cup\front)}
  \big((h^s+\delta^s+h\delta^{s-1})\|\bpsi\|_{\bH^s(\Omega_\Sigma)}+
  h\|\ROT\bpsi\|_{\bL^2(\Omega_\Sigma)}\big).
\end{align*}
We handle $I_2$ by using the Cauchy-Schwarz inequality on every
$\partial\Omega_i$, $i=1,\cdots,N$.
\begin{align*}
  I_2& \le c\,
  h^{-\frac12}\|\jump{\bF_h\CROSS\bn}\|_{\bL^2(\Sigma\cup\front)}
  \sum_{i=1}^N
  h^{\frac12}\|\calK_\delta\bpsi\|_{\bL^2(\partial\Omega_i)}.
\end{align*}
We use~\eqref{eq:L2Gamma_2} on every $\Omega_i$ with
$\Theta:=\frac{1-2s}{2(1-s)}$ to obtain
\begin{align*}
  I_2 & \le c\, h^{-\frac12}\|\jump{\bF_h\CROSS\bn}\|_{\bL^2(\Sigma\cup\front)}
\sum_{i=1}^Nh^{\frac12}\|\calK_\delta\bpsi\|_{\bH^s(\Omega_i)}^{1-\Theta} 
\|\calK_\delta\bpsi\|_{\bH^{1}(\Omega_i)}^{\Theta} \\
  & \le
  c\, h^{-\frac12}\|\jump{\bF_h\CROSS\bn}\|_{\bL^2(\Sigma\cup\front)}
  h^{\frac12}\|\calK_\delta\bpsi\|_{\bH^s(\Omega_\Sigma)}^{1-\Theta}
  \|\calK_\delta\bpsi\|_{\bH^{1}(\Omega_\Sigma)}^{\Theta},
\end{align*}
where the constant $c$ depends on $N$, which we recall is a fixed
number.  Using again the approximation properties of $\calK_\delta$ we
infer that
\begin{align*}
  I_2 & \le c\,
  h^{-\frac12}\|\jump{\bF_h\CROSS\bn}\|_{\bL^2(\Sigma\cup\front)}
  h^{\frac12}\delta^{(s-1)\Theta}\|\bpsi\|_{\bH^s(\Omega_\Sigma)} \\
  &\le
  c\, h^{-\frac12}\|\jump{\bF_h\CROSS\bn}\|_{\bL^2(\Sigma\cup\front)}
  h^{\frac12}\delta^{s-\frac12}\|\bpsi\|_{\bH^s(\Omega_\Sigma)}.
\end{align*}
Then \eqref{eq:approx_b_lambda} is obtained by gathering
the above estimates and setting $\delta=h$.
\end{proof}
\begin{rem}[Alternative Decomposition]
  Estimate \eqref{eq:approx_b_lambda} can alternatively be derived using the
  decomposition $\bpsi= \bpsi-\calC_h^l\bpsi+\calC_h^l\bpsi$ instead of
  $\bpsi= \bpsi-\calK_\delta\bpsi+\calK_\delta\bpsi$.
\end{rem}

\begin{thm}\label{thm:cvg_norme_l2}
  Let $\bg\in\bL^2(\Omega)$ and let $(\bE,p)$ be the solution
  of~\eqref{bvp}. Let $\tau<\min(\tau_\eps,\tau_\kappa)$ where
  $\tau_\eps$ and $\tau_\kappa$ are defined in
  Theorem~\ref{Thm:Hs_stability_bvp}. Let $(\bE_h,p_h)$ be
  solution of~\eqref{e:discrete_formulation}. For
  any $\alpha\in\left(\frac{\ell(1-\tau)}{\ell-\tau},1\right)$, there
  exists $c>0$, uniform in $h$, such that
\begin{equation}\label{eq:cvg_norme_l2_min_reg}
  \|\bE-\bE_h\|_{\bL^2(\Omega)} \le c\,h^{r_1+r_2}\|\bg\|_{\bL^2(\Omega)},
\end{equation}
{with $r_1:=\min\left(1-\alpha,\alpha-1
  +\tau\left(1-\frac{\alpha}{\ell}\right)\right)$ and $r_2= r_1$ if
$\DIV(\eps\bg) \ne 0$ and $r_2 = \alpha-1
+\tau\left(1-\frac{\alpha}{\ell}\right)$ if $\DIV(\eps\bg) =0$.}  If
in addition $\bE\in\bH^{k+1}(\Omega_\Sigma)$ and $p\in
\xHn{k+\alpha}(\Omega_\Sigma)$ for some $0<k<\ell-1$, then the following
holds:
\begin{equation}\label{eq:cvg_norme_l2_smooth}
\|\bE-\bE_h\|_{\bL^2(\Omega)} \le c\, h^{k+r_1}\left(\|\bg\|_{\bL^2(\Omega)}
+\|\bE\|_{\bH^{k+1}(\Omega_\Sigma)} + \|p\|_{\xHn{k+\alpha}(\Omega_\Sigma)} \right).
\end{equation}
\end{thm}
\begin{proof}
  We are going to use a duality argument \`a la Nitsche-Aubin.  In the
  following we denote $a_h^1$ the extension to
  $\left[(\bZ^\tau(\Omega)+\bX_h)\CROSS \xHone_0(\Omega)\right]^2$ of the
  bilinear form defined on $\left[\bX_h\CROSS M_h\right]^2$
  in~\eqref{def_a_h} with $\theta=1$. Then the following symmetry
  property holds:
\begin{align*}
a_h^1\left((\bF,q),(\bG,r)\right) =
  a_h^1\left((\bG,-r),(\bF,-q)\right).
\end{align*}
for all $((\bF,q),(\bG,r))\in\left[(\bZ^\tau(\Omega)+\bX_h)\CROSS
  \xHone_0(\Omega)\right]^2$.  Let $(\bw,q)\in\Hzcurl\CROSS
\xHone_0(\Omega)$ be the solution of the following (adjoint) problem:
\begin{equation}
\ROT(\kappa\ROT\bw) - \eps\GRAD q = \eps\left(\bE-\bE_h\right), \qquad
\DIV(\eps \bw)=0. \label{adjoint_problem}
\end{equation}
Recall that Theorem~\ref{Thm:Hs_stability_bvp} implies that
$\bw\in\bZ^\tau(\Omega)\cap\bH^\tau(\Omega)$ and that
\begin{align}
  \|\bw\|_{\bH^\tau(\Omega)}+\|\kappa\ROT\bw\|_{\bH^\tau(\Omega)}+
  \|\ROT\kappa\ROT\bw\|_{\bL^2(\Omega)} &\le c\,
  \|\bE-\bE_h\|_{\bL^2(\Omega)}. \label{eq:prop_w_1}
\end{align}
\bal{Upon testing
  \eqref{adjoint_problem} with $\bE-\bE_h$, using the definition of
  $a_h^1$ in \eqref{def_a_h}, and recalling that $\DIV(\eps\bw)=0$ and both the
  tangential jump of $\bw$ across $\Sigma$ and the tangential trace of
  $\bw$ on $\front$ are zero, we obtain the following identity:} 
\begin{align*}
  \|\eps^{\frac12}(\bE&-\bE_h)\|_{\bL^2(\Omega)}^2 =
  a_h^1\left((\bw,-q),(\bE-\bE_h,p_h-p)\right)
  + c_\alpha h^{2(1-\alpha)} \left(\eps\GRAD q,\GRAD(p_h-p)\right)_\Omega
\end{align*}
The definition of the pair $(\bw,q)$ implies that
$(\eps\GRAD q,\GRAD\varphi)_\Omega
=-(\eps(\bE-\bE_h,\GRAD\varphi)\bal{)_\Omega}$
  for all $\varphi\in \xHone_0(\Omega)$; hence,
\begin{align*}
   \|\eps^{\frac12}(\bE-\bE_h)\|_{\bL^2(\Omega)}^2 & =a_h^1\left((\bE-\bE_h,p-p_h),(\bw,q)\right)
  + c_\alpha h^{2(1-\alpha)}\left(\eps(\bE-\bE_h),\GRAD(p-p_h)\right)_\Omega \\
  & = a_h\left((\bE-\bE_h,p-p_h),(\bw,q)\right)
  +c_\alpha h^{2(1-\alpha)}\left(\eps(\bE-\bE_h),\GRAD(p-p_h)\right)_\Omega \\
  & \quad +(1-\theta)\left(\mean{\kappa\ROT\bw},
    \jump{-\bE_h\CROSS\bn}\right)_{\Sigma\cup\front}.
\end{align*}
We now use the Galerkin
orthogonality and we introduce the global approximation
$\calC_h^g\calK_\delta\bw$, with $\delta=h^{1-\frac{\alpha}{\ell}}$,
and the pressure approximation $\calC_h^p q$:
\begin{multline}
  \|\eps^{\frac12}(\bE-\bE_h)\|_{\bL^2(\Omega)}^2
  = a_h\left((\bE-\bE_h,p-p_h),(\bw-\calC_h^g\calK_\delta\bw,q-\calC_h^p q)\right) \\
  + c_\alpha
  h^{2(1-\alpha)}\left(\eps(\bE-\bE_h),\GRAD(p-p_h)\right)_\Omega -
  (1-\theta)\left(\kappa\ROT\bw,\jump{\bE_h\CROSS\bn}\right)_{\Sigma\cup\front}.
\label{provisional_L2_estimate}
\end{multline}
Note that we replaced $\mean{\kappa\ROT\bw}$ by $\kappa\ROT\bw$ since
the tangent component of $\kappa\ROT\bw$ is continuous across the
interfaces owing to $\ROT(\kappa\ROT\bw)\in \bL^2(\Omega)$.

We now handle the three terms in the right hand side separately. For
the first one, we use Proposition~\ref{prop:cont_bis} with
$s=1-\alpha$, $\bF=\bw$ and $\bF_h=\calC_h^g\calK_\delta\bw$ (note
that $\bF_h\in \bY_h\subset \bX_h \cap \Hzcurl$); we then infer that
\begin{align*}
  \big|a_h\big((\bE-\bE_h&, p-p_h),(\bw-\calC_h^g\calK_\delta\bw,q-\calC_h^p q)\big)\big| \le \\
  &c\, \|\bE-\bE_h,p-p_h\|_h\big(\|\bw-\calC_h^g\calK_\delta\bw,q-\calC_h^p q\|_h \\
  &+ h^{\alpha-1}\|\bw-\calC_h^g\calK_\delta\bw\|_{\bL^2(\Omega)} +
  h^{-\alpha}\|q-\calC_h^p q\|_{\xLtwo(\Omega)} +
  h^{\frac12-\alpha}\|q-\calC_h^p q\|_{\xLtwo(\Sigma)}
\\
  &+ h\|\ROT\kappa\ROT(\bw-\calC_h^g\calK_\delta\bw)\|_{\bL^2(\calT_h)} 
+h^{1-\alpha}\|\ROT(\bw-\calC_h^g\calK_\delta\bw)\|_{\bH^{1-\alpha}(\Omega)}\big).
\end{align*}
The \bal{term in parentheses on the right-hand side is estimated as} in the proof of
Theorem~\ref{thm:cvg_norme_h_min_reg}. We then have
\begin{equation}\label{eq:cvg_norme_l2_1}
\begin{aligned}
  \big|a_h\big((\bE-\bE_h,p-p_h),&(\bw-\calC_h^g\calK_\delta\bw,\calC_h^p q-q)\big)\big| \\
  & \le c\, \|\bE-\bE_h,p-p_h\|_h h^{r_1}\|\bE-\bE_h\|_{\bL^2(\Omega)}.
\end{aligned}
\end{equation}
The second term in \eqref{provisional_L2_estimate} is estimated by
using the Cauchy-Schwarz inequality, the definition of the norm
$\|\cdot\|_h$ and the inequality $r_1\le 1-\alpha$,
\begin{align}
  \left|h^{2(1-\alpha)}\left(\eps(\bE-\bE_h),\GRAD(p-p_h)\right)_\Omega\right|
  &\le c \, h^{1-\alpha}\|\GRAD(p-p_h)\|_{\bL^2(\Omega)} h^{1-\alpha}\|\bE-\bE_h\|_{\bL^2(\Omega)} \nonumber \\
  &\le c\, \|\bE-\bE_h,p-p_h\|_h
  h^{r_1}\|\bE-\bE_h\|_{\bL^2(\Omega)}. \label{eq:cvg_norme_l2_2}
\end{align}
The last term in \eqref{provisional_L2_estimate} is estimated by using
Lemma~\ref{lem:approx_b_lambda} with $\bpsi:=\kappa\ROT\bw$ and
$s:=\tau$:
\begin{align}
  \big|(1-\theta)\big(\kappa\ROT\bw&,
\jump{\bE_h\CROSS\bn}\big)_{\Sigma\cup\front}\big| \nonumber \\
  &\le c\, \|\bE-\bE_h,\bal{p-p_h}\|_h \left(h^\tau\|\kappa\ROT\bw\|_{\bH^{\tau}(\Omega)}
+h\|\ROT(\kappa\ROT\bw)\|_{\bL^{2}(\Omega)}\right)\nonumber \\
  &\le c\,
  \|\bE-\bE_h,\bal{p-p_h}\|_h\, h^{r_1}\|\bE-\bE_h,\bal{p-p_h}\|_{\bL^2(\Omega)},\label{eq:cvg_norme_l2_3}
\end{align}
where we have used~\eqref{eq:prop_w_1} and $r_1\le \frac{\tau}{2}<\tau$.
Upon inserting
\eqref{eq:cvg_norme_l2_1}-\eqref{eq:cvg_norme_l2_2}-\eqref{eq:cvg_norme_l2_3}
in \eqref{provisional_L2_estimate} we obtain
\[
\|\eps^{\frac12}(\bE-\bE_h)\|_{\bL^2(\Omega)}^2
\le ch^{r_1}\|\bE-\bE_h\|_{\bL^2(\Omega)} \|\bE-\bE_h,p-p_h\|_h.
\]
Owing to the uniform positivity of $\eps$, this leads to:
\[
\|\bE-\bE_h\|_{\bL^2(\Omega)} \le ch^{r_1}\|\bE-\bE_h,p-p_h\|_h.
\]
Now we consider two cases.  Assuming only minimal regularity,
Theorem~\ref{thm:cvg_norme_h_min_reg} gives a bound on
$\|\bE-\bE_h,p-p_h\|_h$ that leads to \eqref{eq:cvg_norme_l2_min_reg}.
If $\bE$ and $p$ are piecewise smooth, then we can apply
Theorem~\ref{thm:cvg_norme_h_smooth} and we obtain
\eqref{eq:cvg_norme_l2_smooth}.
\end{proof}

\begin{rem}\label{rem:optimality}
  Let $\tau\in(0,\frac12)$ and denote $(\bE,p)$ the solution of
  \eqref{bvp}. Assume that $\bE\in\bH^{\tau}(\Omega)$ and
  $\bE\notin\bH^{\tau^+}(\Omega)$ for all $\tau^+>\tau$. Then,
  irrespective of the value of $\DIV(\eps\bg)$, the best choice for
  $\alpha$ is $\alpha = \frac{\ell(2-\tau)}{2\ell-\tau}$, which gives
  the convergence rate $r_1+r_2 =
  \tau\frac{\ell-1}{\ell-\frac{\tau}{2}}$; this convergence rate
  approaches the optimal rate, $\tau$, when the approximation degree
  $\ell$ is large. Note also that $\alpha$ is close to $1$ when $\ell$
  is large.
\end{rem}

\begin{rem}
  Note that the degree of the polynomials used for $M_h$ is not
  involved in the convergence rate when minimal regularity is
  assumed. This means that we can use different degrees of polynomials
  for $\bX_h$ and $M_h$, and that it is sufficient to take polynomials
  of degree $1$ for $M_h$ to get convergence.
\end{rem}

\subsection{Numerical illustrations}
In this section we illustrate numerically the performance of the
method on a boundary value problem on the $L$-shaped domain
\[
\Omega=(-1,1)^2\backslash\left([0,+1]\CROSS[-1,0]\right).
\] 
We assume that $\Omega$ is composed of three subdomains:
\[
\Omega_1 = (0,1)^2,\qquad \Omega_2 = (-1,0)\CROSS(0,1),\qquad \Omega_3
= (-1,0)^2.
\]
We use $\kappa\equiv1$ in $\Omega$, $\eps_{|\Omega_2} = 1$ and
$\eps_{|\Omega_1} = \eps_{|\Omega_3} =:\eps_r$. Denoting $\lambda>0$ a
real number such that $
\tan\left(\frac{\lambda\pi}{4}\right)\tan\left(\frac{\lambda\pi}{2}\right)
=\eps_r, $ we define the scalar potential $S_\lambda(r,\vartheta) =
r^{\lambda}\phi_{\lambda}(\vartheta)$, where $(r,\vartheta)$ are the polar
coordinates, and $\phi_\lambda$ is defined by
\[
  \phi_\lambda(\vartheta) = 
  \begin{cases} \sin(\lambda\vartheta) &\textnormal{ if
    }0\le\vartheta<\frac{\pi}{2}, \\
    \frac{\sin\left(\frac{\lambda}{2}\pi\right)}
    {\cos\left(\frac{\lambda}{4}\pi\right)}
    \cos\left(\lambda\left(\vartheta-\frac34 \pi\right)\right)
    &\textnormal{ if
    }\frac{\pi}{2}\le\vartheta<\pi,\\
    \sin\left(\lambda\left(\tfrac{3}{2}\pi-\vartheta\right)\right)
    &\textnormal{ if }\pi\le\vartheta\le\frac{3\pi}{2}.
\end{cases}
\]
Then we solve the problem
\begin{equation}
\ROT\ROT\bE = 0,\qquad \DIV(\eps\bE) 
= 0,\qquad \bE\CROSS\bn_{|\partial\Omega} = \GRAD S_\lambda\CROSS\bn.
\label{bvp_test}
\end{equation}
The exact solution is
$\bE = \GRAD S_\lambda {\in\bH^\lambda(\Omega)}$.  \bal{We present in
  Table~\ref{Table1} two series of simulations done with the
  two-dimensional version of SFEMaNS, see \eg \cite{GLLNR11}, with
  $\polP_1$ finite elements on quasi-uniform triangular Delaunay
  meshes; \ie $\ell=2$ in \eqref{def_of_X_h}. All the technical
  assumptions made in the paper are met: \eqref{Hyp:eps_mu} hold and the
  meshes are quasi-uniform and $H^1$-conforming.} We use $\lambda=0.535$
in Table~\ref{tab:cvg_bvp_0.535} and $\lambda=0.24$ in
Table~\ref{tab:cvg_bvp_0.24}, which gives $\eps_r\simeq0.5$ and
$\eps_r\simeq7.55\ 10^{-2}$, respectively.  The relative error in the
$\bL^2$-norm is reported in the column ``rel. err.''  and the
convergence rate is reported in the column ``coc''. Several values of
$\alpha$ are used to evaluate the effect of $\lambda$ and $\alpha$ on
the convergence rates.  We observe that the convergence rate is
quasi-optimal when $\alpha$ is close to $1$, which is consistent with
Remark~\ref{rem:s_less_14}, since \eqref{bvp_test} can be re-written
in the form \eqref{bvp} with $\DIV(\eps\bg)=0$.

\begin{table}[h]\caption{$\bL^2$-errors and convergence rates {with $\ell=2$}. The
    convergence rates are almost optimal for $\alpha=0.9$ in both
    cases. }
\label{Table1}\vspace{-.5\baselineskip}
\centering
\subtable[$\bL^2$-errors and convergence rates for $\lambda=0.535$]{%
\begin{tabular}{|c||c|c||c|c||c|c||c|c|}\hline
\multirow{2}{*}{$h$} & \multicolumn{2}{c||}{$\alpha=0.4$} 
& \multicolumn{2}{c||}{$\alpha=0.6$} & \multicolumn{2}{c||}{$\alpha=0.9$} & \multicolumn{2}{c|}{$\alpha=1.0$}\\
\hhline{~--------}
       &rel. err. &coc  &rel. err. &coc &rel. err. &coc &rel. err. &coc\\ \hline
0.2    & 2.332E-1 &  -  & 1.444E-1 &  - & 1.249E-1 & -  &1.297E-1  & - \\
0.1    & 2.473E-1 &-0.08& 1.168E-1 &0.31& 8.846E-2 &0.50&9.167E-2  & 0.50 \\
0.05   & 2.631E-1 &-0.09& 9.452E-2 &0.31& 6.186E-2 &0.52&6.392E-2  & 0.52 \\
0.025  & 2.797E-1 &-0.09& 7.700E-2 &0.30& 4.289E-2 &0.53&4.427E-2  & 0.53 \\
0.0125 & 2.968E-1 &-0.09& 6.312E-2 &0.29& 2.962E-2 &0.53&3.059E-2  & 0.53 \\
\hline
\end{tabular}\label{tab:cvg_bvp_0.535}
}\\
\subfigure[$\bL^2$-errors and convergence rates for $\lambda=0.24$]{
\begin{tabular}{|c||c|c||c|c||c|c||c|c|}\hline
\multirow{2}{*}{$h$} & \multicolumn{2}{c||}{$\alpha=0.4$} 
& \multicolumn{2}{c||}{$\alpha=0.6$} & \multicolumn{2}{c||}{$\alpha=0.9$} & \multicolumn{2}{c|}{$\alpha=1.0$}\\
\hhline{~--------}
       &rel. err. &coc  &rel. err. &coc &rel. err. &coc  &rel. err. &coc   \\ \hline
0.2    & 5.773E-1 & -   & 4.739E-1 &  -  & 4.426E-1 &  - &4.495E-1  & -    \\
0.1    & 6.209E-1 &-0.11& 4.507E-1 &0.07& 3.801E-1 &0.22 &3.838E-1  & 0.23 \\
0.05   & 6.711E-1 &-0.11& 4.413E-1 &0.03& 3.259E-1 &0.22 &3.272E-1  & 0.23 \\
0.025  & 7.180E-1 &-0.10& 4.452E-1 &-0.01& 2.788E-1 &0.23&2.788E-1  & 0.23 \\
0.0125 & 7.564E-1 &-0.08& 4.602E-1 &-0.05& 2.380E-1 &0.23&2.376E-1  & 0.23 \\
\hline
\end{tabular}\label{tab:cvg_bvp_0.24}
}
\end{table}
 {It has been pointed out in the literature (see \eg
  \cite[\S8.3.1]{CosDau02}, \cite{MR2485453}, \cite{MR2914268}) that
  it is possible to build special meshes allowing the existence of
  $\xCone$ interpolation operators, \ie it is possible to represent
  gradients on these meshes with optimal approximation properties.  We
  now investigate theses possibilities with $\polP_1$ and $\polP_2$
  finite elements. We solve again the above boundary value problem
  with $\lambda=0.535$ and $\alpha=0.9$. For the $\polP_1$
  approximation, we construct Powell-Sabin type meshes (see
  \cite{PowSa:77}) and compare the results obtained on these meshes
  with those obtained on generic Delaunay meshes (see
  Table~\ref{tab:cvg_bvp_powell_sabin}). We indeed observe an improvement since now the
  convergence rate is optimal, \ie close to $0.535$.  For the
  $\polP_2$ approximation we construct Hsieh-Clough-Tocher meshes, see
  \cite[item 4, p.~520]{HCT_1966}. It is possible to construct on
  these meshes $\polP_3$ finite element spaces containing $\xCone$
  functions with optimal approximation properties. Then, the standard
  vector-valued $\polP_2$ finite element spaces constructed on these
  meshes contains enough gradients. We compare the results obtained on
  Hsieh-Clough-Tocher meshes with those obtained on generic Delaunay
  meshes (see Table~\ref{tab:cvg_bvp_hct}). We do not observe any significant
  improvement, since the optimal order was already numerically
  achieved on the generic Delaunay meshes.}
\begin{table}[h]\caption{$\bL^2$-errors and convergence rates for $\lambda=0.535$, $\alpha=0.9$ on different kinds of meshes }
\label{Table2}\vspace{-.5\baselineskip}
\centering
\subtable[$\bL^2$-errors and convergence rates for $\ell=1$]{%
\begin{tabular}{|c||c|c||c|c|}\hline
\multirow{2}{*}{$h$}&\multicolumn{2}{c||}{Delaunay mesh}&\multicolumn{2}{c|}{Powell-Sabin mesh} \\
\hhline{~----}
       & rel.err    & coc  & rel. err  & coc  \\ \hline
0.2    &  2.166E-1  &  -   & 1.742E-1  &  -   \\
0.1    &  1.652E-1  & 0.39 & 1.246E-1  & 0.48 \\
0.05   &  1.268E-1  & 0.38 & 8.711E-2  & 0.52 \\
0.025  &  9.821E-2  & 0.37 & 6.052E-2  & 0.53 \\
0.0125 &  7.758E-2  & 0.34 & 4.200E-2  & 0.53 \\ \hline
\end{tabular}\label{tab:cvg_bvp_powell_sabin}
}

\subtable[$\bL^2$-errors and convergence rates for $\ell=2$]{%
\begin{tabular}{|c||c|c||c|c|}\hline
\multirow{2}{*}{$h$}&\multicolumn{2}{c||}{Delaunay mesh}&\multicolumn{2}{c|}{Hsieh-Clough-Tocher mesh} \\
\hhline{~----}
 & rel.err & coc & rel. err & coc \\ \hline
0.2    & 1.297E-1 &  -   & 1.359E-1 &  -   \\
0.1    & 9.167E-2 & 0.50 & 9.446E-2 & 0.53 \\
0.05   & 6.392E-2 & 0.52 & 6.535E-2 & 0.53 \\
0.025  & 4.427E-2 & 0.53 & 4.515E-2 & 0.53 \\
0.0125 & 3.059E-2 & 0.53 & 3.117E-2 & 0.53 \\ \hline
\end{tabular}\label{tab:cvg_bvp_hct}

}
\end{table}

\section{Eigenvalue problem}\label{s:eigen}
We extend in this section the theory introduced above to eigenvalue
problems. We want to establish an approximation result for the
solutions to the following problem: Find
$(\bE,\lambda)\in[\Hzcurl\cap\Hdiveps]\CROSS\Real$ such that
\begin{equation}
\ROT\kappa\ROT\bE = \lambda\eps\bE.
\end{equation}
We restrict ourselves in the rest of this section to the symmetric
variant of the bilinear form $a_h$ defined in \eqref{def_a_h}, \ie we
set $\theta=1$.  We finally assume from now on that $\alpha$ is chosen
as in Theorem~\ref{thm:cvg_norme_h_min_reg}, \ie
\begin{equation}
\alpha\in\left(\frac{\ell(1-\tau)}{\ell-\tau},1\right), \label{alpha_range_for_eigen}
\end{equation} 
where $\tau$
is the minimal regularity index of the problem~\eqref{bvp} as defined
in Theorem~\ref{Thm:Hs_stability_bvp}. In the following we set
$r:=\min\left(1-\alpha,\alpha-1+\tau\left(1-\frac\alpha\ell\right)\right)$.

\subsection{Framework}
Let us equip $\bL^2(\Omega)$ with the inner product
$(\bef,\bg)_\eps:=\int_\Omega\eps \bef\SCAL\bg$. This inner product is
equivalent to the usual $\bL^2$-inner product owing
to~\eqref{Hyp:eps_mu}. The associated norm is denoted
$\|\cdot\|_{\eps}$.

For any $\bg\in\bL^2(\Omega)$, we denote $(\bE,p)$ the solution
of~\eqref{bvp} and we set $A\bg:=\bE$. This defines an operator
$A:\bL^2(\Omega)\rightarrow\bL^2(\Omega)$ that is self-adjoint and
compact (\cf Theorem~\ref{Thm:Hs_stability_bvp}).  We now define two
families of discrete operators $\calE_h: \bL^2(\Omega)\longrightarrow
\bX_h$ and $\calP_h: \bL^2(\Omega)\longrightarrow M_h$ so that for any
$\bg\in\bL^2(\Omega)$, the pair $(\calE_h \bg,\calP_h \bg)$ solves
\eqref{e:discrete_formulation}. Then we finally define
\begin{equation}\begin{aligned}
A_h:\bL^2(\Omega) & \longrightarrow \bX_h + \GRAD M_h \subset
\bL^2(\Omega) \\
\bg & \longmapsto \calE_h \bg - c_\alpha h^{2(1-\alpha)} \GRAD \calP_h \bg.
\end{aligned}
\end{equation} 
\bal{We want to study whether the eigenvalues and eigenspaces spaces of $A_h$ converges to those 
of $A$. For this purpose we are going to use the following result:}
\begin{thm}[Spectral correctness \cite{Bab_Osb_91,Osborn75}]\label{thm:Osborn}
Let $X$ be an Hilbert space and $A:X\rightarrow X$ be a self-adjoint
compact operator. Let $\Theta=\{h_n;\ n\in\polN\}$ be a discrete
subset of $\Real$ such that $h_n\rightarrow 0$ as
$n\rightarrow+\infty$. Assume that there exists a family of operators
$A_h:X\rightarrow X$, $h\in \Theta$, such that:\begin{itemize}
\item  $A_h$ is a linear self-adjoint operator, for all $h\in\Theta$.
\item $A_h$ converges pointwise to $A$.
\item The family is collectively compact.
\end{itemize}
Let $\mu$ be an eigenvalue of $A$ of multiplicity $m$ and let
$\{\phi_j\}$, ${j=1,\cdots,m}$ be a set of associated orthonormal
eigenvectors.
\begin{enumerate}[(i)]
\item For any $\epsilon>0$ such that the disk $B(\mu,\epsilon)$
  contains no other eigenvalues of $A$, there exists $h_\epsilon$ such
  that, for all $h<h_\epsilon$, $A_h$ has exactly $m$ eigenvalues
  (repeated according to their multiplicity) in the disk
  $B(\mu,\epsilon)$.
\item In addition, for $h<h_\epsilon$, if we denote $\mu_{h,j}$,
  $j=1,\cdots,m$ the set of the eigenvalues of $A_h$ in
  $B(\mu,\epsilon)$, there exists $c>0$ such that
\begin{equation}\label{eq:Osborn}
c|\mu-\mu_{h,j}|\le\sum_{j,l=1}^m|\left((A-A_h)\phi_j,\phi_l\right)_X|+\sum_{j=1}^m\|(A-A_h)\phi_j\|_X^2.
\end{equation}
\end{enumerate}
\end{thm}

\subsection{Approximation result}
We start by proving that the operators $\{A_h\}$
are self-adjoint, then we prove the pointwise
convergence, and we finally establish the collective compactness.
\begin{lem}\label{Lem:self_adjointness}
For any $h$, $A_h:\bL^2(\Omega)\rightarrow\bL^2(\Omega)$ is a
self-adjoint operator, \ie for any $\be,\bef\in\bL^2(\Omega)$, the
following holds
\begin{equation}
\left(A_h\be,\bef\right)_\eps = \left(\be,A_h\bef\right)_\eps.
\end{equation}
\end{lem}
\begin{proof}
Let $\be,\bef\in \bL^2(\Omega)$. By definition we have
\[
a_h\left((\calE_h \be ,\calP_h \be),(\calE_h \bef,-\calP_h \bef)\right)
= \left(\be,\calE_h \bef\right)_{\eps} - c_\alpha
h^{2(1-\alpha)}\left(\be,\GRAD \calP_h \bef\right)_\eps =
\left(\be,A_h\bef\right)_\eps.
\]
Using the symmetry properties of $a_h$, we infer
\begin{align*}
a_h\left((\calE_h \be ,\calP_h \be),(\calE_h \bef,-\calP_h \bef)\right) &=
a_h\left((\calE_h \bef,\calP_h \bef),(\calE_h \be ,-\calP_h \be\right)) \\ &=
\left(\bef,\calE_h \be\right)_{\eps} -c_\alpha h^{2(1-\alpha)}\left(\bef,\GRAD
\calP_h \be\right)_\eps = \left(\bef,A_h\be\right)_\eps,
\end{align*}
thereby proving that the operator $A_h$ is self-adjoint on the Hilbert
space $\bL^2(\Omega)$ equipped with the inner product
$\left(\cdot,\cdot\right)_\eps$.
\end{proof}

\begin{lem}\label{Lem:L2_convergence}
Under the above assumptions, there exists $c>0$, uniform with respect
to $h$ such that,
\begin{equation}
\forall\be\in\bL^2(\Omega),\qquad \|A_h\be-A\be\|_{\eps}\le ch^{2r}\|\be\|_{\eps}. \label{strong_convergence}
\end{equation}
\end{lem}
\begin{proof}
Let $A\be\in\bL^2(\Omega)$ and $p\in \xHone_0(\Omega)$ such that
$\ROT(\kappa\ROT A\be) + \eps\GRAD p = \eps\be$. 
Using the triangular inequality,
Theorems~\ref{thm:cvg_norme_h_min_reg} and~\ref{thm:cvg_norme_l2}, the
equivalence between the norms on $\bL^2(\Omega)$ and the fact that
$r\le 1-\alpha$, we infer that
\begin{align*}
\|A\be&-A_h\be\|_{\eps} \le \|A\be-\calE_h \be\|_{\eps} + c_\alpha
h^{2(1-\alpha)}\|\GRAD \calP_h \be -\GRAD p\|_{\eps} +
h^{2(1-\alpha)}\|\GRAD p\|_{\eps} \\ &\le c(h^{2r}\|\be\|_\eps +
h^{1-\alpha}\|A\be-\calE_h \be,p-\calP_h \be\|_h + h^{2(1-\alpha)}\|\be\|_{\eps})
\le c\,h^{2r}\|\be\|_{\eps},
\end{align*}
which concludes the proof.
\end{proof}
Note that the above result is stronger than the pointwise convergence
hypothesis, \ie $A_h$ converges in norm to $A$. Now let us turn our
attention to the question of collective compactness. Recall that a set
$\calA:=\{A_h\in \calL(X;X),\ h\in\Theta\}$ is said to be collectively
compact if, for each bounded set $U\subset X$, the image set $\calA
U:=\left\{A_h\bg,\; \bg\in U,\; A_h\in\calA\right\}$ is relatively
compact in $X$.
\begin{lem}\label{Lem:collective_compactness}
The family $\{A_h\}_{h>0}$ is collectively compact under the above
assumptions provided $\alpha\in\left(\frac{\ell(1-\tau)}{\ell-\tau},1\right)$.
\end{lem}
\begin{proof} \bal{Of course, Lemma~\ref{Lem:L2_convergence} implies the
  result, but we are now going to provide an alternative proof.}  Owing to
  the compact embedding $\bH^s(\Omega)\subset\bL^2(\Omega)$ for any
  $s>0$, it is sufficient to prove that there exists $s>0$ and $c>0$
  such that, for any $\bg\in\bL^2(\Omega)$ and any $h>0$,
\[
\|A_h\bg\|_{\bH^s(\Omega)}\le c\|\bg\|_{\bL^2(\Omega)}.
\]
Let us take $\bg\in\bL^2(\Omega)$. Owing to the definition of $\bX_h$
and $M_h$, we know that $A_h\bg\in\bH^s(\Omega)$ for any
$s\in\left(0,\frac12\right)$.  Moreover, there exists $c$, only
depending on $s$ and the shape regularity of the mesh sequence, such
that the following inverse inequality holds:
\[
\|A_h\bg\|_{\bH^{s}(\Omega)} \le ch^{-s}\|A_h\bg\|_{\bL^2(\Omega)}.
\]
Let us consider \bal{$s<\min(r,\frac{\tau}{2})$}. Using the triangular
inequality, interpolation results, the above inverse inequality
together with Theorems~\ref{thm:cvg_norme_l2}
and~\ref{Thm:Hs_stability_bvp} \bal{and Lemma~\ref{Lem:L2_convergence}} leads to:
\begin{align*}
  \|A_h\bg\|_{\bH^{s}(\Omega)} &\le
  \|A_h\bg-A\bg\|_{\bH^{s}(\Omega)}+\|A\bg\|_{\bH^{s}(\Omega)} \\ &\le
  c\, \|A_h\bg-A\bg\|_{\bL^{2}(\Omega)}^{\frac12}
  \|A_h\bg-A\bg\|_{\bH^{2s}(\Omega)}^{\frac12}+c\, \|\bg\|_{\bL^2(\Omega)}
  \\ &\le c\,
  h^r\|\bg\|_{\bL^2(\Omega)}^{\frac12}\left(h^{-s}\|A_h\bg\|_{\bL^2(\Omega)}^{\frac12}+\|A\bg\|_{\bH^{2s}(\Omega)}^{\frac12}\right)+c\, \|\bg\|_{\bL^2(\Omega)}
  \\ &\le c\, \left(h^{r-s}+1\right)\|\bg\|_{\bL^2(\Omega)}.
\end{align*}
This implies the collective compactness of $\{A_h\}$ since $r>s$.
\end{proof}
We conclude that the approximation is spectrally correct, \ie we can apply
Theorem~\ref{thm:Osborn} by combining
Lemmas~\ref{Lem:self_adjointness}, \ref{Lem:L2_convergence},
\ref{Lem:collective_compactness}. \bal{Note finally that the convergence rate on the eigenvalues
is at least $\calO(h^{2r})$ owing to \eqref{eq:Osborn} and  \eqref{strong_convergence}.}

\subsection{Numerical illustration {for $\alpha<1$}}
In this section, we present some eigenvalues computations. We consider
the square $\Omega=(-1,1)^2$ in the plane. We divide $\Omega$ into
four subdomains
\[
\Omega_1 = (0,1)^2,\qquad \Omega_2 = (-1,0)\CROSS(0,1), \qquad
\Omega_3 = (-1,0)^2,\qquad \Omega_4= (0,1)\CROSS(-1,0).
\]
We use $\kappa\equiv1$ in $\Omega$,
$\eps_{|\Omega_1}=\eps_{|\Omega_3}=1$ and
$\eps_{|\Omega_2}=\eps_{|\Omega_4}=\eps_r$. Benchmark results for this
checkerboard problem are available in \cite{website:Dauge} for
$\eps_r^{-1}\in\{2,10,100,10^8\}$. Tables~\ref{tab:cvg_vp_eps_r_2}
and~\ref{tab:cvg_vp_eps_r_10} show results for $\eps_r=0.5$ and
$\eps_r=0.1$ respectively. The ratio
$\frac{|\lambda_c-\lambda_r|}{\lambda_r}$ is reported in column
``rel. err.'', where $\lambda_c$ and $\lambda_r$ are the computed and
reference eigenvalues, respectively. The reference values are those
from the benchmark. The computed order of convergence is shown in the
column ``coc''. The computations have been done using ARPACK (\cf
\cite{MR1621681}) with tolerance $10^{-8}$.  Note that the computed
order of convergence seems to reach a constant value for sufficiently
small $h$, for every eigenvalue, as expected.
\begin{table}[h]
\centering
\caption[Approximation of the first four eigenvalues 
for $\eps_r=0.5$]{Approximation of the first four 
eigenvalues for $\eps_r=0.5$. We used $\alpha=0.7$ in the simulations.}
\label{tab:cvg_vp_eps_r_2}
\begin{tabular}{|c||c|c||c|c||c|c||c|c|}
\hline
\multirow{2}{*}{$h$} & \multicolumn{2}{c||}{$\lambda_r\simeq3.3175$}& \multicolumn{2}{c||}{$\lambda_r\simeq3.3663$}& \multicolumn{2}{c||}{$\lambda_r\simeq6.1863$}& \multicolumn{2}{c|}{$\lambda_r\simeq13.926$}\\
\hhline{~--------}
& rel. err. & coc& rel. err. & coc & rel. err. & coc & rel. err. & coc \\ \hline
0.2    & 9.364E-4 & - & 3.943E-3 & - & 1.439E-1 & - & 6.104E-1 & -\\
0.1    & 1.833E-4 &2.35& 2.147E-3 &0.88& 1.734E-4 &9.70& 4.484E-1 &0.44  \\
0.05   & 3.751E-5 &2.29& 1.188E-3 &0.85& 2.241E-5 &2.95& 1.599E-1 &1.49  \\
0.025  & 8.405E-6 &2.16& 6.463E-4 &0.88& 2.833E-6 &2.98& 1.120E-5 &13.8  \\
0.0125 & 2.081E-6 &2.01& 3.439E-4 &0.91& 3.667E-7 &2.95& 1.478E-6 &2.92  \\
\hline
\end{tabular}
\end{table}
\begin{table}[h]
\centering
\caption[Approximation of the first four eigenvalues for $\eps_r=0.1$]{Approximation of the first four eigenvalues for $\eps_r=0.1$. We used $\alpha=0.8$ in the simulations.}
\label{tab:cvg_vp_eps_r_10}
\begin{tabular}{|c||c|c||c|c||c|c||c|c|}
\hline
\multirow{2}{*}{$h$} & \multicolumn{2}{c||}{$\lambda_r\simeq4.5339$}& \multicolumn{2}{c||}{$\lambda_r\simeq6.2503$}& \multicolumn{2}{c||}{$\lambda_r\simeq7.0371$}& \multicolumn{2}{c|}{$\lambda_r\simeq22.342$}\\
\hhline{~--------}
& rel. err. & coc& rel. err. & coc & rel. err. & coc & rel. err. & coc \\ \hline
0.2    & 4.559E-1 & - & 6.052E-1 & - & 6.410E-1 & - & 8.869E-1 & - \\
0.1    & 2.859E-1 &0.67& 4.731E-1 &0.36& 5.310E-1 &0.27& 8.512E-1 &0.06   \\
0.05   & 3.306E-2 &3.11& 2.982E-1 &0.67& 3.763E-1 &0.50& 8.033E-1 &0.08   \\
0.025  & 2.154E-6 &13.9& 7.748E-2 &1.94& 1.772E-1 &1.09& 7.406E-1 &0.12   \\
0.0125 & 2.608E-7 &3.05& 3.258E-3 &4.57& 5.946E-7 &18.2& 6.602E-1 &0.17   \\
\hline
\end{tabular}
\end{table}

\subsection{{The case $\alpha=1$}}
{We have shown that the numerical method is optimally convergent
  with $\alpha=1$ for the boundary value problem \eqref{bvp} if
  $\DIV(\eps\bg)=0$. It is then reasonable to investigate the
  convergence properties of the method for the eigenvalue problem with
  $\alpha=1$ even though the theoretical analysis seems to show that
  there might be a loss of compactness in this case; \ie we cannot
  apply Theorem~\ref{thm:Osborn}. We investigate this issue by solving
  again the checkerboard problem introduced in the previous section
  and by comparing the results obtained with $\alpha=0.7$ and
  $\alpha=1$. We compute the first 10 eigenvalues for $\eps_r=0.5$
  and report the results in Table~\ref{Table:alpha1_P1} for $\polP_1$
  finite elements and Table~\ref{Table:alpha1_P2} for $\polP_2$ finite
  elements. The typical meshsize in these simulations is $0.025$.
  Inspection of these tables show that the approximation with
  $\alpha=1$ is not spectrally correct. Other results on meshes with
  different meshsizes or structure (Delaunay, Powell-Sabin or HCT),
  not reported here, show the same type of behavior, \ie there are
  spurious eigenvalues when $\alpha=1$.  This series of numerical
  tests confirms the sharpness on the upper bound on $\alpha$ stated
  in Lemma~\ref{Lem:collective_compactness}.
\begin{table}[h]
\centering
\caption[Approximation of the first ten eigenvalues for $\eps_r=0.5$]{{Approximation} of the first ten eigenvalues with $\polP_1$ elements and $\eps_r=0.5$. Comparison between $\alpha=0.7$ and $\alpha=1.0$.}
\label{Table:alpha1_P1}
\begin{tabular}{|c||c|c||c|c|}\hline
\multirow{2}{*}{$\lambda$} & \multicolumn{2}{c||}{$\alpha=0.7$} & \multicolumn{2}{c|}{$\alpha=1.0$} \\
\hhline{~----}
       & app. value & rel. error  & app. value & rel. error \\ \hline
 3.31755 &  3.31844 &  2.70E-4 & 3.31790 &  1.06E-4 \\
 3.36632 &  3.37816 &  3.51E-3 & 3.36786 &  4.56E-4 \\
 6.18639 &  6.18732 &  1.50E-4 & 3.91497 &  3.67E-1 \\    
 13.9263 &  13.9321 &  4.14E-4 & 3.91616 &  7.18E-1 \\    
 15.0830 &  15.0888 &  3.88E-4 & 4.14335 &  7.25E-1 \\    
 15.7789 &  15.7859 &  4.48E-4 & 4.29445 &  7.27E-1 \\    
 18.6433 &  18.6555 &  6.53E-4 & 4.30863 &  7.68E-1 \\    
 25.7975 &  25.8163 &  7.29E-4 & 15.0191 &  4.17E-1 \\    
 29.8524 &  29.8684 &  5.36E-4 & 35.7192 &  1.96E-1 \\    
 30.5379 &  30.5643 &  8.66E-4 & 305.349 &  9.00E0 \\    
\hline
\end{tabular}
\end{table}

\begin{table}[h]
\centering
\caption[Approximation of the first ten eigenvalues for $\eps_r=0.5$]{{Approximation} of the first ten eigenvalues with $\polP_2$ elements and $\eps_r=0.5$. Comparison between $\alpha=0.7$ and $\alpha=1.0$.}
\label{Table:alpha1_P2}
\begin{tabular}{|c||c|c||c|c|}\hline
\multirow{2}{*}{$\lambda$} & \multicolumn{2}{c||}{$\alpha=0.7$} & \multicolumn{2}{c|}{$\alpha=1.0$} \\
\hhline{~----}
       & app. value & rel. error  & app. value & rel. error \\ \hline
 3.31755 &  3.31758  &  8.55E-6  &  3.31756 &  2.30E-6  \\
 3.36632 &  3.36857  &  6.68E-4  &  3.36634 &  3.62E-6  \\
 6.18639 &  6.18641  &  3.14E-6  &  4.28879 &  3.07E-1  \\    
 13.9263 &  13.9265  &  1.05E-5  &  4.29153 &  6.92E-1  \\    
 15.0830 &  15.0832  &  1.14E-5  &  4.30113 &  7.15E-1  \\    
 15.7789 &  15.7791  &  1.36E-5  &  4.30145 &  7.27E-1  \\    
 18.6433 &  18.6436  &  1.52E-5  &  4.30683 &  7.69E-1  \\    
 25.7975 &  25.7979  &  1.36E-5  &  12.8213 &  5.03E-1  \\    
 29.8524 &  29.8530  &  2.04E-5  &  37.1980 &  2.46E-1  \\    
 30.5379 &  30.5395  &  5.43E-5  &  1308.73 &  4.19E+1  \\    
\hline
\end{tabular}
\end{table}
}
\section*{Acknowledgments} We are pleased to acknowledge fruitful
discussions with Simon Labrunie and Patrick Ciarlet. \bal{We also thank the
referees for their thorough reviews and helpful comments which lead to
many improvements in the manuscript.}

\appendix 
\section{Technical Lemmas}
Let $\{\calT_h\}_{h>0}$ be an affine shape-regular mesh sequence in
$\Real^3$. Let $T_K:\wK\longrightarrow K$ be the affine
mapping that maps the reference element $\wK$ to $K$ and let $J_K$ be
the Jacobian of $T_K$. It is a standard result that there are
constants that depend only on $\wK$ and the shape regularity constants
of the mesh sequence so that
\begin{equation}
  \|J_K\| \le c h_K,\quad \|J_K^{-1}\| \le c h_K^{-1}, 
  \qquad |\det(J_K)| \le c h_K^3, \quad |\det(J_K^{-1})| \le c h_K^{-3},
\end{equation} 
where $h_K$ is the diameter of $K$.

\begin{lem} \label{Lem:Hs_zero_average}For all $s\in [0,1]$, there is
  a constant $c$, uniform with respect to the mesh sequence, so that the
  following holds for all cells $K\in\calT_h$ and all $\psi\in \xHn{s}(K)$
  with zero average over $K$:
\begin{equation}
  \|\wpsi\|_{\xHn{s}(\wK)} \le  c h_K^{s-\frac{d}{2}} \|\psi\|_{\xHn{s}(K)}
  ,\qquad\text{where}\qquad  \wpsi(\bx) := \psi(T_K(\bx))
\end{equation}

\end{lem}
\begin{proof} Upon making the change of variable $\bx=T_K(\wbx)$ we
  obtain
\[
\|\wpsi\|_{\xLtwo(\wK)} = |\det(J_K)|^{-\frac12} \|\psi\|_{\xLtwo(K)} \le c
h_K^{-\frac{d}{2}} \|\psi\|_{\xLtwo(K)}.
\]
Likewise, using the fact that $\wpsi$ is of zero average,
the Poincar\'e
inequality implies
\begin{align*}
  \|\wpsi\|_{\xHone(\wK)} & = \bal{\left(\|\wpsi\|_{\xLtwo(\wK)}^2 + \|\wGRAD
    \wpsi\|_{\xLtwo(\wK)}^2\right)^{\frac12} \le (c_p(\wK) +1)^{\frac12}
  \|\wGRAD \wpsi\|_{\xLtwo(\wK)} }\\
  &\le c |\det(J_K)|^{-\frac12} \|J_K\| \|\GRAD \psi\|_{\xLtwo(K)} \le c
  h_K^{-\frac{d}{2}+1} \|\psi\|_{\xHone(K)}.
\end{align*} 
Then, the \bal{interpolation} theorem implies
that
\[
\|\wpsi\|_{\dot{\xHn{s}}(\wK)} \le c\, h_K^{s-\frac32} \|\psi\|_{\dot{\xHn{s}}(K)},
\]
where we defined $\dot{\xHn{s}}(E) :=
[\dot \xLtwo(E),\dot{\xHone}(E)]_{s,2}$ with $\dot \xLtwo(E)$ and
$\dot{\xHone}(E)$ being the subspaces of the functions of zero
average in $\xLtwo(E)$ and $\xHone(E)$, respectively. We conclude using
Lemma~\ref{Lem:Hs_average}
\end{proof}

\begin{lem} \label{Lem:Hs_average} The spaces
  $[\dot \xLtwo(E),\dot{\xHone}(E)]_s$ and $[\xLtwo(E),\xHone(E)]_s\cap
  \dot \xLtwo(E)$ are identical and the induced norms are identical,
  \ie $\|v\|_{\dot{\xHn{s}}(E)} = \|v\|_{\xHn{s}(E)}$ for all $v\in
  [\xLtwo(E),\xHone(E)]_s\cap \dot  \xLtwo(E)$.
\end{lem}
\begin{proof}
  One can use Lemma A1 from \cite{Guer_LBB_2007} with $T$ being the
  projection onto $\dot \xLtwo(\Omega)$.
\end{proof}

We now state the main result of this section.  It is a variant of
Lemma~8.2 in \cite{MR2263045} with the extra term
$\|\bphi\|_{\bL^2(K)}$.  Our proof slightly differs from that in
\cite{MR2263045} since the proof therein did not appear convincing to
us (actually, the embedding inequality at line 9, page 2224 in
\cite{MR2263045} has a constant that depends on the size of the cell;
for instance, using \bal{a constant vector field for $\bphi$} in this
inequality yields a contradiction.  As result the estimate (8.11) in
\citep{MR2263045} is not uniform with respect to $h$).
\begin{lem} \label{l:cont_IP}For all $k\in \polN$ and all $\sigma\in
  (0,\frac12)$ there is $c$, uniform with respect to the mesh sequence,
  so that the following holds for all faces $F\in \calF_h$ in the
  mesh, all polynomial function $\bv$ of degree at most $k$, and all
  function $\bphi\in\bH^\sigma(K)\cap \bH(\text{\rm curl},K)$
\begin{equation}\label{eq:cont_IP_1}
 \bal{ \Big|\int_F (\bv{\times}\bn)\SCAL \bphi\Big| }\le 
  c \|\bv\|_{\bL^2(F)}h_F^{-\frac12} (h_K^\sigma \|\bphi\|_{\bH^\sigma(K)} 
  + h_K \|\ROT\bphi\|_{\bL^2(K)} + 
  \|\bphi\|_{\bL^2(K)}), 
\end{equation}
where $K$ is either one of the two elements sharing the face $F$.
\end{lem}

\begin{proof} 
  We restrict ourselves to three space dimensions.  In two space
  dimensions $\bphi$ is scalar-valued and the proof must be modified
  accordingly. Let $K$ be either one of the two elements sharing the
  face $F$.  Let $\obphi$ be the average of $\bphi$ over $K$ and let
  us denote $\bpsi:=\bphi-\obphi$.  Upon denoting $\wbv(\wbx)= J_K^T
  \bv(T_K(\wbx))$ and $\wbpsi(\wbx) = J_K^T \bpsi(T_K(\wbx))$, it is a
  standard result (see \cite[3.82]{MR2059447}) that
\[
\int_F (\bv{\times}\bn)\SCAL \bpsi  = \int_\wF (\wbv{\times}\wbn)\SCAL \wbpsi,
\]
where $\wbn$ is one of the two unit normals on $\wF$.  Let us extend
$\wbv$ by zero on $\partial \wK{\setminus}\wF$; then $\wbv \in
\bH^{\frac12-\sigma}(\partial \wK)$ for all $\sigma>0$, \bal{since the extension by zero
is stable in the $H^s$-norm for all $s\in[0,\frac12)$, see \eg \cite[Thm.~11.4]{LM68} for smooth domains
and \cite[Thm.~1.4.2.4 or Cor.~1.4.4.5]{bkGrisvard} for Lipschitz domains.} Note that it
is not possible to have $\sigma=0$. Now let
$R:\bH^{\frac12-\sigma}(\partial \wK)\longrightarrow
\bH^{1-\sigma}(\wK)$ be a standard lifting operator. There is a
constant depending only on $\wK$ and $\sigma$ so that
\[
\| R \wbv\|_{\bL^2(\wK)} + \|\wROT R\wbv\|_{\bH^{-\sigma}(\wK)} \le
c(\wK,\sigma) \|R\wbv\|_{\bH^{1-\sigma}(\wK)} \le c'c(\wK,\sigma)
\|\wbv\|_{\bH^{\frac12-\sigma}(\wF)},
\]
where $\wROT$ is the curl operator in the coordinate system of $\wK$.
Then, {slightly abusing the notation by using integrals instead of duality products,
we have}
\begin{align*}
\left|  \int_\wF (\wbv{\times}\wbn)\SCAL \wbpsi \right| & 
= \left|  \int_\wK \left((R\wbv)\SCAL \wROT \wbpsi - \wbpsi\SCAL \wROT (R\wbv)\right)\right| \\
  & 
  \le c \left( \|(R\wbv)\|_{\bL^2(\wK)} \|\wROT \wbpsi
    \|_{\bL^2(\wK)} +
    \|\wbpsi\|_{\bH^{\sigma}_0(\wK)} \|\wROT (R\wbv)\|_{\bH^{-\sigma}(\wK)}\right) \\
  & \le c \left( \|\wROT \wbpsi \|_{\bL^2(\wK)} +
    \|\wbpsi\|_{\bH^{\sigma}_0(\wK)} \right)
  \|\wbv\|_{\bH^{\frac12-\sigma}(\wF)} \\
 & \le c \left( \|\wROT \wbpsi \|_{\bL^2(\wK)} +
    \|\wbpsi\|_{\bH^{\sigma}(\wK)} \right)
  \|\wbv\|_{\bH^{\frac12-\sigma}(\wF)},
\end{align*}
where we used that $\bH^{\sigma}(\wK)=\bH^{\sigma}_0(\wK)$ for
$\sigma\in [0,\frac12)$.  Due to norm equivalence for discrete
functions over $\wK$ and using that $\|J_K\|\le c h_K$, $h_K/h_F\le c$
and $|F| \le c h_F^2$ in three space dimensions, where $c$ depends of
the shape-regularity constant of the mesh sequence and the polynomial
degree $k$, we have
\[
\|\wbv\|_{\bH^{\frac12-\sigma}(\wF)} \le c \|\wbv\|_{\bL^2(\wF)} \le
c \|J_K\| |F|^{-\frac12} \|\bv\|_{\bL^2(F)} \le c h_K h_F^{-1}
\|\bv\|_{\bL^2(F)}\le c'\|\bv\|_{\bL^2(F)}.
\]
Using the identity (see \cite[Cor.~3.58]{MR2059447})
\[
(\ROT\bpsi)(T_K(\wbx)) = \frac{1}{\text{det}(J_K)} J_K (\wROT\wbpsi)(\wbx),
\]
we obtain 
\[
\|\wROT\wbpsi\|_{\bL^2(K)} \le c |\text{det}(J_K)|^{\frac12
}\|J_K^{-1}\|\|\ROT\bpsi\|_{\bL^2(K)} \le c h_K^{\frac12 }\
\|\ROT\bpsi\|_{\bL^2(K)}.
\]
Since the average of $\bpsi$ over $K$ is zero, we can use
Lemma~\ref{Lem:Hs_zero_average} (with an extra scaling by $\|J_K\|$ for $\wbpsi = J_K^T \bpsi(T_K)$) to
deduce
\[
\|\wbpsi\|_{\bH^\sigma(\wK)} \le c h_K^{\sigma-\frac12}   \|\bpsi\|_{\bH^\sigma(K)}.
\]
In conclusion we have obtained the following estimate:
\[
\int_F (\bv{\times}\bn)\SCAL (\bphi-\obphi) \le c \left( h_K\
  \|\ROT\bphi\|_{\bL^2(K)} + h_K^{\sigma}
  \|\bphi-\obphi\|_{\bH^\sigma(K)}\right) h_K^{-\frac12}
\|\bv\|_{\bL^2(F)}.
\]
Observing that $\|1\|_{\bH^\sigma(K)} \le \|1\|_{\bL^2(K)}^{1-\sigma}
\|1\|_{\bH^1(K)}^{\sigma}= \|1\|_{\bL^2(K)}=|K|^{\frac12}$, we infer
that
\[
\|\bphi-\obphi\|_{\bH^\sigma(K)} \le \|\bphi\|_{\bH^\sigma(K)} +|\obphi| |K|^{\frac12}
\]
The Cauchy-Schwarz inequality yields
$|\obphi| \le |K|^{-\frac12}\|\bphi\|_{\bL^2(K)}$; as a result,
\[
\|\bphi-\obphi\|_{\bH^\sigma(K)} \le \|\bphi\|_{\bH^\sigma(K)} +
\|\bphi\|_{\bL^2(K)} \le 2 \|\bphi\|_{\bH^\sigma(K)}.
\]
Now we evaluate a bound from above on $\int_F (\bv{\times}\bn)\SCAL
\obphi$ as follows: 
\begin{align*}
  \left|\int_F (\bv{\times}\bn)\SCAL \obphi\right| & \le |\obphi|
  |F|^{\frac12}\|\bv\|_{\bL^2(F)} \le |K|^{-\frac12}\|\bphi\|_{\bL^2(K)}|F|^{\frac12}\|\bv\|_{\bL^2(F)} \\
  & \le c \|\bv\|_{\bL^2(F)} h_F^{-\frac12} \|\bphi\|_{\bL^2(K)}.
\end{align*}
The result follows by combining all the above estimates.
\end{proof}

\begin{lem} \label{Lem:L2Gamma}
Let $\alpha\in(\frac12,1)$.
There is exists a constant $c(\alpha)$ so that
\begin{equation}\label{eq:L2Gamma_1}
  \|u\|_{\xLtwo(\front)} \le c(\alpha) \|u\|_{\xLtwo(\Omega)}^{1-\frac{1}{2\alpha}} 
  \|u\|_{\xHn{\alpha}(\Omega)}^{\frac{1}{2\alpha}}, \qquad \forall u\in \xHn{\alpha}(\Omega).
\end{equation}
Similarly, for $s\in\left(0,\frac12\right)$, there exists a constant
$c(s)$ so that, for $\Theta:=\frac{1-2s}{2(1-s)}$,
\begin{equation}\label{eq:L2Gamma_2}
  \|u\|_{\xLtwo(\front)} \le c(s) \|u\|_{\xHn{s}(\Omega)}^{1-\Theta} 
  \|u\|_{\xHone(\Omega)}^{\Theta}, \qquad \forall u\in \xHone(\Omega).
\end{equation}
\end{lem}

\begin{proof}
We start with the standard estimate
\[
\|u\|_{\xLtwo(\front)} \le c \|u\|_{\xLtwo(\Omega)}^{\frac12}
\|u\|_{\xHone(\Omega)}^{\frac12}, \qquad \forall u\in \xHone(\Omega),
\]
which allows us to apply Lemma~\ref{Lem:Lions-Petree}. This implies
that the trace operator is a continuous linear mapping from
$[\xLtwo(\Omega),\xHone(\Omega)]_{\frac12,1}$ to $\xLtwo(\front)$.  Then the \bal{Lions-Petree
reiteration theorem~\cite[Thm.~26.3]{bkTartar_2007}} implies that
\begin{align*} [\xLtwo(\Omega),\xHn{\alpha}(\Omega)]_{\frac{1}{2\alpha},1} &=
  [\xLtwo(\Omega),[\xLtwo(\Omega),\xHone(\Omega)]_{\alpha,2}]_{\frac{1}{2\alpha},1}
  = [\xLtwo(\Omega),\xHone(\Omega)]_{\frac12,1} \\
  [\xHn{s}(\Omega),\xHone(\Omega)]_{\Theta,1} &=
  [[\xLtwo(\Omega),\xHone(\Omega)]_{s,2},\xHone(\Omega)]_{\Theta,1} =
  [\xLtwo(\Omega),\xHone(\Omega)]_{\frac12,1}
\end{align*}
The norms being equivalent, we can eventually write:
\begin{align*}
  \|u\|_{\xLtwo(\front)} \le
  c\|u\|_{[\xLtwo(\Omega),\xHone(\Omega)]_{\frac12,1}} \le
  c(\alpha)\|u\|_{[\xLtwo(\Omega),\xHn{\alpha}(\Omega)]_{\frac{1}{2\alpha},1}}
  \le c(\alpha)\|u\|_{\xLtwo(\Omega)}^{1-\frac1{2\alpha}}
  \|u\|_{\xHn{\alpha}(\Omega)}^{\frac1{2\alpha}},\\
  \|u\|_{\xLtwo(\front)} \le
  c\|u\|_{[\xLtwo(\Omega),\xHone(\Omega)]_{\frac12,1}}\le
  c(s)\|u\|_{[\xHn{s}(\Omega),\xHone(\Omega)]_{\Theta,1}} \le
  c(s)\|u\|_{\xHn{s}(\Omega)}^{1-\Theta}\|u\|_{\xHone(\Omega)}^{\Theta}.
\end{align*}
This concludes the proof.
\end{proof}

\begin{lem}[Lions-Petree] \label{Lem:Lions-Petree} 
Let $E_1\subset E_0$ be two Banach spaces, with continuous embedding. Let $L$ be a linear mapping $E_1\rightarrow F$ with $F$ another Banach space. For $s\in(0,1)$, $L$ extends to a linear mapping from $[E_0,E_1]_{s,1}$ to $F$ if and only if there exists $C>0$ such that
\[
\forall u\in E_1,\qquad \|Lu\|_{F} \le C\|u\|_{E_0}^{1-s}\|u\|_{E_1}^s.
\]
\end{lem}
\begin{proof}
See Lemma 25.3 in \cite{bkTartar_2007}.
\end{proof}

\

\bibliographystyle{abbrvnat}

\end{document}